%% file: equiv.tex
\documentclass[11pt]{article}
\usepackage[utf8x]{inputenc}
\usepackage[a4paper]{geometry}
\usepackage{amsmath}
\usepackage{amsfonts}
\usepackage{amsthm}
\usepackage{amssymb}
\usepackage{tikz}
\usepackage{tikz-cd}
\usepackage{indentfirst}

\theoremstyle{plain}
\newtheorem{thm}{Theorem}
\newtheorem{lem}[thm]{Lemma}
\newtheorem{prop}[thm]{Proposition}
\newtheorem{cor}[thm]{Corollary}
\theoremstyle{definition}
\newtheorem{defn}[thm]{Definition}
\newtheorem*{thm*}{Theorem}

\theoremstyle{remark}
\newtheorem{ex}[thm]{Example}

\newcommand{\mc}{\mathcal}
\newcommand{\mbb}{\mathbb}
\newcommand{\mrm}{\mathrm}
\newcommand{\mf}{\mathfrak}
\newcommand{\ol}{\overline}
\newcommand{\vtheta}{\vartheta}

\renewcommand{\AA}{\mbb A}
\newcommand{\PP}{\mbb P}

\newcommand{\NN}{\mbb N}
\newcommand{\ZZ}{\mbb Z}
\newcommand{\QQ}{\mbb Q}
\newcommand{\RR}{\mbb R}
\newcommand{\CC}{\mbb C}
\newcommand{\oo}{\circ}
\newcommand{\SL}{\mrm{SL}}

\newcommand{\epi}{\twoheadrightarrow}
\renewcommand{\emptyset}{\varnothing}
\DeclareMathOperator{\cl}{cl}
\DeclareMathOperator{\sss}{ss}
\DeclareMathOperator{\End}{End}
\DeclareMathOperator{\Gal}{Gal}
\DeclareMathOperator{\an}{an}
\DeclareMathOperator{\pr}{pr}
\DeclareMathOperator{\Par}{Par}
\DeclareMathOperator{\Bor}{Bor}
\DeclareMathOperator{\Hilb}{Hilb}
\DeclareMathOperator{\diag}{diag}
\DeclareMathOperator{\Spec}{Spec}
\DeclareMathOperator{\Stab}{Stab}
\DeclareMathOperator{\opp}{opp}
\DeclareMathOperator{\Hom}{Hom}
\DeclareMathOperator{\Mon}{Mon}

\usepackage{hyperref}	

\title{Wonderful compactifications of Bruhat-Tits buildings in the non-split case}
\author{Dorian Chanfi}

\begin{document}
	\input{intro.tex}

	\input{galois_wc.tex}
	\input{berko.tex}
	\input{fixed_points.tex}
	\bibliographystyle{halpha-abbrv}
	\bibliography{equiv}
\end{document}

%% file: intro.tex
\maketitle

\begin{abstract}
    Given an adjoint semisimple group $G$ over a local field $k$, we prove that the maximal Satake-Berkovich compactification of the Bruhat-Tits building of $G$ can be identified with the one obtained by embedding the building into the Berkovich analytification of the wonderful compactification of $G$, extending previous results of Rémy, Thuillier and Werner. In the process, we use the characterisation of the wonderful compactification in terms of Hilbert schemes given by Brion to extend the definition of the wonderful compactification to the case of a non-necessarily split adjoint semisimple group over an arbitrary field and investigate some of its properties pertaining to rational points on the boundary.
    
    Lastly, given a finite possibly ramified Galois extension $k'/k$, we take a look at the action of the Galois group on the maximal compactification of the building of $G$ over $k'$ and check that the Galois-fixed points are precisely the limits of sequences of fixed points in the building over $k'$, though they may not lie in the Satake-Berkovich compactification of $G$ over $k$. 
\end{abstract}

\tableofcontents

\newpage
\section*{Introduction}

In this article, we study compactifications of Bruhat-Tits buildings and relate them to compactifications of algebraic groups. Bruhat-Tits buildings were first introduced in the work of Iwahori and Matsumoto \cite{IM} in the late 1960s, then in a series of articles of Bruhat and Tits \cite{BT1}, \cite{BT2} with the aim of studying geometrically the structure of reductive groups over local, or more generally non-archimedean valued, fields. An affine building is a polysimplicial complex covered by subcomplexes called apartments, each of which being isomorphic to a common Euclidean tiling, subject to some very restrictive incidence relations \cite[IV.1]{Brown}. 
Bruhat-Tits theory associates to each reductive group over a local field an affine building on which the group of rational points acts strongly transitively in the sense of \cite[V.1]{Brown}. This building plays the role of a coarse analogue to a symmetric space for reductive groups over valued fields: It is endowed with a nonpositively curved metric that is invariant under the group action and it provides a parameterisation for maximal compact subgroups through its vertices.

This guiding analogy motivates the search for compactifications of affine buildings, as compactifications of symmetric spaces have played an important role in group theory -- as in the proof of Mostow's strong rigidity theorem -- and in number theory, for example via the study of locally symmetric spaces that appear naturally as moduli spaces of algebro-geometric objects of interest.
In \cite{GR}, the authors provide some group-theoretic payoff by establishing that the set of vertices of the so-called polyhedral compactification of the Bruhat-Tits building of a semisimple group $G$ over a local field $k$ parameterises closed amenable subgroups of $G(k)$ with connected Zariski closure that are maximal for these properties.

This compactification is the main object of study of this article. First constructed by Landvogt \cite{Landvogt} by a glueing procedure similar to the construction of the building itself, it was incorporated by Werner \cite{Werner} into an ordered family of compactifications of which it is the maximal element. This family was recovered by Rémy, Thuillier and Werner in \cite{RTW1} using Berkovich geometry and proved equivalent to the previous construction in \cite{RTW2}. The construction proceeds first by embedding the building $\mc B(G, k)$ into the Berkovich analytic space $G^{\an}$ associated to $G$, then mapping the result to the analytic space $\Bor(G)^{\an}$ associated to the maximal flag variety of $G$. The boundary of the compactification given by the closure of the image then decomposes as the union of the Bruhat-Tits buildings of the semisimple quotients of all the parabolic $k$-subgroups of $G$.

In \cite{RTW3}, the authors establish that, if $G$ is a split semisimple adjoint group defined over a local field $k$, then the polyhedral compactification $\ol{\mc B}(G, k)$ -- or maximal Satake-Berkovich compactification as it is called in that context -- of the building $\mc B(G,k)$ embeds naturally in the Berkovich analytification $\ol G^{\an}$ of the wonderful compactification $\ol G$ of $G$. They prove a compatibility result between the topological stratification of the polyhedral compactification into smaller affine buildings and the algebro-geometric stratification of $\ol G$ into $(G \times G)$-orbits.

The wonderful compactification was introduced in \cite{DCP} for split semisimple adjoint groups defined over $\mathbb C$ and extended to arbitrary algebraically closed fields in \cite{Strickland} and \cite{DCS}. Among its remarkable features, it admits a finite stratification into $(G \times G)$-orbits $X(\tau)$ indexed by the types $\tau$ of parabolic subgroups of $G$ and these orbits fibre over products of flag varieties corresponding to opposite parabolic subgroups.

The upshot of \cite{RTW3} is that boundary strata of $\ol{\mc B}(G, k)$ corresponding to parabolic subgroups of type $\tau$ are mapped to the analytification $X(\tau)^{\an}$ of $\ol G$.

In this article, we extend this result to the case of a non-split group. Namely, we prove that the maximal Satake-Berkovich compactification of a semisimple adjoint group $G$ over a local field $k$ can be mapped homeomorphically to the analytification of the wonderful compactification $\ol G$ of $G$ and that the compatibility between boundaries is preserved. 

The article splits in three largely independent sections, if one is ready to take the existence the wonderful compactification of a non-split group on faith.

In the first section, we provide a suitable extension of the definition of the wonderful compactification to the case of a general semisimple adjoint group.  We prove that, given a semisimple adjoint group $G$ over a field $k$ and $l$ a finite extension that splits $G$, the wonderful compactification of $G_l$ is endowed with an action of $\Gal(l/k)$ by automorphisms and therefore descends to $k$. The existence of this action follows from the description given by Brion \cite{Bri} of the wonderful compactification in terms of Hilbert schemes. We also study the $(G \times G)$-orbits and prove that the compactification has no rational points on the boundary if and only if $G$ is anisotropic.

In the second section, we establish the following theorem.

\begin{thm*}
    Let $G$ be a semisimple adjoint group over a non-archimedean field $k$. Denote by $\ol G$ its wonderful compactification, by $\mc B(G, k)$ the Bruhat-Tits building of $G$ over $k$ and $\ol{\mc B}(G, k)$ its maximal Satake-Berkovich compactification.
    \begin{enumerate}
        \item There exists a continuous $G(k) \times G(k)$-equivariant map $\ol{\Theta}: \mc B(G, k) \times \ol{\mc B}(G, k) \to \ol{G}^{\an}$ such that, for each $x \in \mc B(G,k)$, the map $$\ol{\Theta}(x, \cdot): \ol{\mc B}(G, k) \to \ol{G}^{\an}$$ is a $G(k)$-equivariant embedding.
        \item For each $x \in \mc B(G, k)$ and any parabolic subgroup $P$ of type $\tau$, the map $\ol{\Theta}(x, \cdot)$ maps the boundary stratum $\mc B(P/R(P), k)$ to the analytic space $X(\tau)^{\an}$ associated to the algebraic boundary stratum $X(\tau)$ of $\ol G$.
        \item If $k$ is locally compact, the embedding $\ol{\Theta}(x, \cdot)$ is a homeomorphism onto the closure of the image of $\mc B(G, k)$ in $\ol G^{\an}$ under $\ol{\Theta}(x, \cdot)$.
    \end{enumerate}
\end{thm*}

The proof follows the general template of \cite{RTW3}. The main difficulty of the proof is the first point, specifically the injectivity of the map which is established in propositions 10 to 13. The existence of the map $\ol{\Theta}$ as well as point 2. in our case become formal corollaries of their counterparts in \cite{RTW3} and are established in proposition 10. Point 3. is established in Corollary 15. 

In the third section, that is essentially independent from the rest of the text, we prove the following result:
\begin{thm*}
If $k$ is a local field and $k'/k$ is a finite Galois extension, then the set of fixed points of $\ol{\mc B}(G, k')$ under the Galois action is precisely the closure of the set of fixed points of the building $\mc B(G,k)$ under the Galois action. 
\end{thm*}
In other words, taking fixed points in the compactification yields no worse results than taking fixed points in the building. The main tool is the description of the compactification $\ol{\mc B}(G, k')$ as a union of Bruhat-Tits buildings. In the process, we make crucial use of metric properties of affine buildings. We also make some general reminders about ramification and fixed points under the Galois action in a Bruhat-Tits building and give some explicit examples of "barbs", that is points in the fixed locus of $\mc B(G, k')$ that do not belong to $\mc B(G, k)$. We prove by an example that the same phenomenon may occur at the boundary of the building.

\subsection*{Notations}

Unless otherwise specified, algebraic group over $k$ will mean (smooth) affine algebraic group scheme over $k$.

We shall use the following notations:
\begin{itemize}
\setlength\itemsep{0em}
\item $k$ is a field.
\item $G$ is a connected semisimple adjoint algebraic group over $k$.
\item $S$ is a maximal split torus in $G$.
\item $_k \Phi = \Phi(G, S)$ is the root system of $G$ with respect to $S$.
\item $W(G, S) = N_G(S)/Z_G(S)$ is the Weyl group of $_k\Phi$.
\item If $X^*(S) = \Hom(S, \mbb G_m)$ is the character group of $S$.
\item If $H$ is an algebraic group over $k$, then $\mc R_u(H)$ is its unipotent radical, that is its maximal smooth connected normal unipotent subgroup.
\item $X_*(S) = \Hom(\mbb G_m, S)$ is the cocharacter group of $S$.
\item $\langle \cdot, \cdot \rangle: X^*(S) \times X_*(S) \to \ZZ$ is the canonical pairing: For $\alpha \in X^*(S), \lambda \in X_*(S)$, we have $\alpha \circ \lambda: t \mapsto t^{\langle \alpha, \lambda\rangle}$. 
\item Given a cocharacter $\lambda$, we let $P_G(\lambda)$ be the parabolic subgroup associated with $\lambda$ \cite[13.d]{Milne}, characterised functorially by the condition that, for each $k$-algebra $A$, we have $$P_G(\lambda)(A) = \{g \in G(A), \lambda(t)g\lambda(t)^{-1} \text{ has a limit as } t \to 0\}.$$
\item $\Par(G)$ is the variety of parabolic subgroups of $G$. \cite[XXVI, 3.5]{SGA3}
\item Recall that, if $k^s$ denotes the separable closure of $k$, then the set of connected components of $\Par(G)$ is in bijection with the set of $\Gal(k^s/k)$-stable sets of vertices of the Dynkin diagram of $G_{k^s}$ \cite[XXVI, 3]{SGA3}. If $(T, B)$ is a Killing couple for $G_{k^s}$, then the vertices of the Dynkin diagram are in one-to-one correspondence with the basis $\Delta$ of the root system $\Phi(G_{k^s}, T)$ associated to $B$. The corresponding action of the Galois group is called the $\ast$-action and is described more concretely below. 
\item If $P$ is a parabolic $k$-subgroup, then we denote by $t(P)$ its connected component, which we will usually think of as a set of roots of $G_{k^s}$ with an action of $\Gal(k^s/k)$ and call its type.
\item We call non-archimedean field any field that is complete with respect to a non-archimedean absolute value.
\item If $H$ is a reductive group over a non-archimedean field $k$, then we denote by $\mc B^e(H, k)$ its extended Bruhat-Tits building and by $\mc B(H, k)$ its reduced Bruhat-Tits building, which can be thought of as the extended Bruhat-Tits building associated to the derived group $[H,H]$. The two notions coincide for semisimple groups.
\end{itemize}

%% file: galois_wc.tex
\section{Wonderful compactifications of semisimple groups}

In this section, we summarise the facts that we shall need about wonderful compactifications of semisimple groups and extend known results about the $(G \times G)$-orbits of wonderful compactifications to the more general setting of an arbitrary adjoint (not necessarily split) semisimple group over a field $k$. We recall the classical representation-theoretic construction given in \cite{DCP} and \cite{Strickland} in the split case and use the alternative, more intrinsic, construction given by Brion in \cite{Bri} to give a definition of the wonderful compactification $\ol G$ in our more general setting.

\subsection{The representation-theoretic construction}\label{classical_dcp}

Let $k$ be a field and $G$ a $k$-split adjoint semisimple group. Fix a $k$-split maximal torus $T$ and $B$ a Borel subgroup containing $T$. Let $(V, \rho)$ be an irreducible representation of $G$ with regular highest weight. The wonderful compactification of $G$ is the closure $\ol G$ of the image of the embedding $$\rho: \begin{array}{ccc} G  & \to & \PP(\End(V))\\ x & \mapsto & [\rho(x)] \end{array}.$$

The embedding is $(G \times G)$-equivariant if we let $G \times G$ act on $G$ via $(g, h) \cdot x = gxh^{-1}$, and on $\PP(\End(V))$ via $(g, h) \cdot [\phi] = [\rho(g)\circ \phi \circ \rho(h)^{-1}]$. This makes $\ol G$ into a $(G\times G)$-space. Moreover, it enjoys the following properties:

\begin{thm}[{\cite[3.9-3.10]{DCS}, \cite[Theorem 2.1]{Strickland}}]
	The wonderful compactification of $G$ is a smooth projective $(G \times G)$-variety $X$ such that
	\begin{itemize}
		\item There exists a $(G \times G)$-equivariant open immersion $j:G \hookrightarrow Z$.
		\item The boundary $Z \setminus G$ is a normal crossings divisor. More precisely, it is the union of $r =\dim T$ smooth hypersurfaces $D_1, \dots, D_r$ which intersect transversally. Moreover, for each $I \subset [\![1,r]\!]$, the intersection $\bigcap_{i \in I} D_i$ is the closure of a single $(G \times G)$-orbit. This correspondence establishes a bijection between subsets of $[\![1,r]\!]$ and $(G \times G)$-orbits.
		\item The intersection $\bigcap_{i \in I}D_i$ is the unique closed $(G\times G)$-orbit of $Z$.
	\end{itemize}
	In addition, these properties characterise $(Z, j)$ up to unique $(G \times G)$-equivariant isomorphism.
\end{thm}

In order to formulate our results, we need some additional information about the $(G\times G)$-orbits of $\ol G$, namely the fact that they can be realised as fibrations over generalised flag varieties.
Following \cite[Section 3]{Bri}, recall that, denoting by $\Delta = \{\alpha_1,  \dots, \alpha_r\} \subset \Phi(G, T)$ the set of simple roots associated to the pair $(T, B)$, we may characterise $(G \times G)$-orbits as follows:

For each ${\tau} \subset \Delta$, there exists a unique cocharacter $\lambda_{\tau} \in X_*(T)$ such that, for each $\alpha \in \Delta$: $$\begin{array}{cc}
\langle \alpha, \lambda_{\tau} \rangle = 1 & \text{if } \alpha \in {\tau}, \\
\langle \alpha, \lambda_{\tau} \rangle = 0 & \text{otherwise}.
\end{array}$$
Indeed, as $G$ is adjoint, $\Delta$ is a basis for the character lattice, which ensures the existence of the $\lambda_{\tau}$.
Then, by properness of $\ol G$, the cocharacter $\lambda_{\tau}$ extends to a morphism of varieties $\lambda_{\tau}: \mbb A^1 \to \ol G$ and, setting $e_{(T,B),\tau} = \lambda_{\tau}(0)$, we have a bijection $$ \begin{array}{ccc}
\{\text{Subsets of } \Delta\} & \longleftrightarrow & \{(G \times G)\text{-orbits of } \ol G \} \\
{\tau} & \longmapsto & X(\tau) := (G \times G)\cdot e_{(T,B),\tau}.
\end{array}$$
Moreover, denoting by $P_{\tau} $ the parabolic subgroup defined by $$P_{\tau} = P_G(\lambda_{\tau}) = \{g \in G, \lambda_{\tau}(t)g\lambda_{\tau}^{-1}(t) \text{ has a limit as } t \to 0\},$$ by $P_{\tau}^{\opp}$ its opposite parabolic subgroup with respect to $T$, and by $L_{\tau} = P_{\tau} \cap P_{\tau}^{\opp}$ their common Levi subgroup, the stabiliser $(G \times G)_{e_{(T,B),\tau}}$ is $\mathcal R_u(P_{\tau} \times P_{\tau}^{\opp}) \rtimes \mathrm{diag}(L_{\tau})(Z(L_{\tau}) \times Z(L_{\tau}))$ or, explicitly, $$\{(ux, vy), u \in \mc R_u(P_{\tau}), v \in \mc R_u(P_{\tau}^{\opp}), x, y \in L_{\tau}, xy^{-1} \in Z(L_{\tau})\}.$$ In particular, the orbit $X(\tau)$ fibres over $\Par_{\tau}(G) \times \Par_{{\tau^{\opp}}}(G)$, where ${\tau^{\opp}}$ is the type of $P_{\tau}^{\opp}$, that is the only subset $\sigma \subset \Delta$ such that $P_{\tau}^{\opp}$ is in the conjugacy class of $P_{\sigma}$. This fibration may be realised by the only $G\times G$-equivariant map $$\pi_{\tau}: X(\tau) \to \Par_{\tau}(G) \times \Par_{{\tau^{\opp}}}(G)$$ that maps $e_{(T,B),\tau}$ to $(P_{\tau}, P_{{\tau}}^{\opp})$. This map is well-defined as $(G \times G)_{e_{(T,B),\tau}} \subset P_{\tau} \times P_{\tau}^{\opp}$ and its fibres may be identified to $L_{\tau}/Z(L_{\tau})$.

\subsection{A construction via Hilbert schemes}

In \cite[Section 2]{Bri}, the author gives a different construction of the wonderful compactification based on Hilbert schemes in the case where $k$ is algebraically closed. His construction can in fact be carried out over any field, as we shall see below, and the comparison theorem \cite[Theorem 3]{Bri} can be extended to the case of a split group over a field. 

\subsubsection*{The embedding}

Let $k$ be a field and $G$ an arbitrary adjoint semisimple group. Let $X = \Bor(G)$ the variety of Borel subgroups of $G$ (\cite[XXII, 5.8.3]{SGA3}). Then, the action of $G$ on itself by conjugacy induces an action of $G \times G$ on $X \times X$, and further on the Hilbert scheme $\Hilb(X \times X)$ of $X$.
\begin{lem}
	Denote by $Z$ the closure of the $(G \times G)$-orbit of the diagonal $\Delta_X \in \Hilb(X \times X)(k)$. Then, the orbit map: 
	$$\begin{array}{c c c}
	G \times G  & \longrightarrow &  Z\\
	(g,h) & \longmapsto & (g,h) \cdot \Delta_X
	\end{array}$$
	yields an open immersion $(G \times G)/\mathrm{diag}(G) \hookrightarrow Z$.
\end{lem} 
\begin{proof}
	To establish this, it is sufficient, by \cite[7.17]{Milne}, to check that the (scheme-theoretic) stabiliser $S$ of $\Delta_X$ is indeed $\diag(G)$. First, observe that the inclusion $\diag(G) \subset S$ is clear. For the reverse inclusion, consider any $k$-algebra $A$. Then, if $(g,h) \in S(A)$, we have $(gh^{-1}, 1)\cdot (\Delta_X)_A = (\Delta_X)_A$. In other words, $gh^{-1}$ acts trivially on $X_A$. Because $G$ is adjoint, by \cite[XXII, 5.8.7]{SGA3}, we have $g = h$. Whence, $S = \diag(G)$.
\end{proof}
Precomposing with the $(G \times G)$-equivariant isomorphism $(G \times G)/\mathrm{diag}(G) \simeq G $ given by $(g, h) \mapsto gh^{-1}$, we get a $(G \times G)$-equivariant open immersion $j: G \hookrightarrow Z$.

\subsubsection*{Comparison in the split case}
In \cite[Theorem 3]{Bri}, the author proves that the pair $(Z_{\ol k}, j_{\ol k})$ is isomorphic to the wonderful compactification $(\ol{G_{\ol k}}, \rho_{\ol k})$ of $G_{\ol k}$, as constructed above. Faithfully flat descent then yields the following

\begin{thm}
	Assume that $G$ is split over $k$ and denote by $\rho: G \hookrightarrow \ol G$ the wonderful compactification associated to a Killing couple $(T, B)$ and a regular highest weight $G$-module, as constructed in Theorem 1. Then, there exists a $(G \times G)$-equivariant isomorphism $f: \ol G \to Z$ such that $f \circ \rho = j$.
\end{thm}

\begin{proof}
	The idea is to observe that the construction of the morphism $f$ can be carried out over $k$ and that being an isomorphism is invariant under faithfully flat base change. We briefly recall the construction below and incorporate the descent arguments when necessary.
	
	Consider the closure $\ol B$ of $B$ in $\ol G$. We consider the associated bundle $(G \times G) \times^{B \times B} \ol B$, defined as the quotient of the product $G \times G \times \ol B$ by the right $(B \times B)$-action given by $$(g, g', x)\cdot (b, b') = (gb, g'b', \rho(b)^{-1}x\rho(b')).$$
	Because $G \times G \to X \times X$ is a principal $(B \times B)$-bundle and $\ol B$ admits a $(B \times B)$-equivariant embedding into projective space (by construction), $\ol B \subset \ol G \subset$, this quotient is indeed well-defined and is a scheme by \cite[Proposition 7.1]{GIT}. Because the $(B \times B)$-action on $G \times G$ is free, the projection $$p: (G \times G) \times^{B \times B} \ol B \to X \times X$$ is a locally trivial fibration with fibre $\ol B$. Moreover, because the action map 
	$$\begin{array}{c c c}
	G \times G \times \ol B & \longrightarrow & \ol G\\
	(g, g', x) & \longmapsto & \rho(g)x\rho(g')^{-1}
	\end{array}$$ is invariant under the above $(B \times B)$-action, it factors through a morphism $$\pi: (G \times G) \times^{B \times B} \ol B \to \ol G.$$
	Finally, the product map $$(p, \pi): (G \times G) \times^{B \times B} \ol B \to X \times X \times \ol G$$ is a closed immersion whose image is flat over $\ol G$. Indeed, by \cite[Section 2]{Bri}, the extension $(p, \pi)_{\ol k}$ is a closed immersion whose image is flat over $\ol G_{\ol k}$ and the properties of being flat and a closed immersion are preserved by faithfully flat descent.
	The universal property of Hilbert schemes then yields a $(G \times G)$-equivariant morphism $f: \ol G \to \Hilb(X \times X)$ which, by construction, associates to a $x \in \ol G$ the closed subscheme $$\{(gBg^{-1}, hBh^{-1}), g,h \in G, x \in (g, h)\ol B\},$$ as follows from the description of the image of $(p, \pi)$ in \cite{Bri}. It follows that $\phi(1)$ is precisely the diagonal $\Delta_X$ and, by $(G \times G)$-equivariance, $\phi(G)$ is the $(G \times G)$-orbit of the diagonal in $\Hilb(X \times X)$. Hence, $f$ factors through the closure $Z$ of $(G \times G) \cdot \Delta_X$ in $\Hilb(X \times X)$ and induces a map, which we shall still denote by $f: \ol G \to Z$. By $(G \times G)$-equivariance, we have for any $g \in G$, $$f(\rho(g)) = (g,1)\cdot f(1) = (g, 1)\cdot \Delta_X = j(g).$$
	Finally, Theorem 3 of \cite{Bri}, used in conjunction with the same descent argument as before, proves that $f$ is an isomorphism.
\end{proof}

\subsection{The wonderful compactification of an arbitrary adjoint semisimple group}

The above comparison theorem motivates the following definition.

\begin{defn}
	Let $G$ be a semisimple adjoint group over a field $k$ and let $X = \Bor(G)$. We call wonderful compactification of $G$ and denote by $\ol G$ the closure of the image of the $(G \times G)$-orbit of the diagonal of $X \times X$ in $\Hilb(X \times X)$.
\end{defn}

In the first subsection, we gave a description of the boundary of the wonderful compactification of a split semisimple adjoint group. Here, we shall describe the boundary of the wonderful compactification of a general semisimple adjoint group, which will amount to an exercise in Galois descent. 

Consider a semisimple adjoint group $G$ over a field $k$. Denote by $S$ a maximal $k$-split torus, $P$ a minimal parabolic subgroup containing $S$. By \cite[XIV, 1.1]{SGA3} applied to $Z_G(S)$, we have our disposal a maximal torus $T$ such that $S \subset T \subset P$. Let $l/k$ be a finite Galois extension, with Galois group denoted by $\Gamma$, such that $T_l$ is split. Let $B$ be a Borel subgroup of $G_l$ such that $S_l \subset T_l \subset B \subset P_l$. Observe that $\Gamma$ acts on $\ol{G}_l = \ol G \times_k l$ via the action on the second factor and that the embedding $G_l \hookrightarrow \ol G_l$ is $\Gamma$-equivariant.

\subsubsection*{Galois action on the boundary orbits}

We begin by proving that, under the identification of the $(G_l \times G_l)$-orbits of $\ol{G_l}$ with the subsets of the basis $\Delta$ of the root system $\Phi(G_l, T_l)$ corresponding to the Killing couple $(T_l, B)$ given in 1.1, the action of $\Gamma$ on the $(G_l \times G_l)$-orbits corresponds to the $\ast$-action of $\Gamma$ on $\mc P(\Delta)$. Recall that the Galois group $\Gamma$ acts on $X_l^*(T)$ in two ways. On the one hand, it has a natural action $(\gamma, \alpha) \mapsto \gamma(\alpha)$ by base change. On the other hand, the basis $\Delta$ allows us to define another action of $\Gamma$ -- the so-called $\ast$-action -- as follows: Given a $\gamma \in \Gamma$ and $\chi \in X_l^*(T)$, we set $\gamma \ast \chi = w_{\gamma}(\gamma(\chi))$, where $w_{\gamma}\in W(G,T)(l)$ is the unique element of the Weyl group such that $w_{\gamma}(\gamma(\Delta)) = \Delta$. This action is in fact independent of the basis $\Delta$ and is essentially equivalent to the action of $\Gamma$ on conjugacy classes of maximal proper parabolic subgroups \cite[2.3]{Tits1} \cite[II.2.1]{Satake}.

We now have
\begin{prop}
	The bijection $$ \begin{array}{ccc}
	\{\text{Subsets of } \Delta\} & \longleftrightarrow & \{(G_l \times G_l)\text{-orbits of } \ol G_l \} \\
	\tau & \longmapsto & X(\tau) = (G_l \times G_l)\cdot e_{(T_l,B),\tau}.
	\end{array}$$ given in subsection 1.1 is $\Gamma$-equivariant for the $\ast$-action on $\Delta$.
\end{prop}

\begin{proof}
Let ${\tau} \subset \Delta$ be a type of parabolic subgroups of $G_l$ and let $\lambda_{\tau}$ be the cocharacter associated to ${\tau}$ as in 1.1. Let $\gamma \in \Gamma$. Then, we have, for each $\alpha \in \Delta$: $$\begin{array}{cc}
\langle \alpha, \gamma(\lambda_{\tau}) \rangle = 1 & \text{if } \alpha \in \gamma({\tau}), \\
\langle \alpha, \gamma(\lambda_{\tau})\rangle = 0 & \text{if } \alpha \in \gamma(\Delta) \setminus \gamma({\tau}),
\end{array}$$
for the action of $\Gamma$ by base change on $X^*(T_l)$ and $X_*(T_l)$.
Let $w\in W(G, T)(l)$ be the element of the Weyl group such that $w(\gamma(\Delta)) = \Delta$ and $\dot w \in N_G(T)(l)$ be a lift of $w$ in $G$.
Then, similarly, we have : $$\begin{array}{cc}
\langle \alpha, w(\gamma(\lambda_{\tau}))\rangle = 1 & \text{if } \alpha \in w(\gamma({\tau})) = \gamma \ast {\tau}, \\
\langle \alpha, w(\gamma(\lambda_{\tau}))\rangle = 0 & \text{if } \alpha \in \Delta \setminus (\gamma \ast {\tau}).
\end{array}$$
In other words, we have $w(\gamma(\lambda_{\tau})) = \lambda_{\gamma \ast {\tau}}$. It follows that $$e_{(T,B), \gamma \ast {\tau}} = \lim_{t\to 0} w(\gamma(\lambda_{\tau}))(t) = \lim_{t \to 0} \dot w \gamma(\lambda_{\tau})(t)\dot w^{-1} = (\dot w, \dot w) \cdot \gamma(e_{(T_l,B),\tau}).$$
Hence, $\gamma$ maps the orbit corresponding to ${\tau}$ to that corresponding to $\gamma \ast {\tau}$.
\end{proof}

\subsubsection*{Rational points on the boundary}\label{fibrations_dcp}

In addition, the fibrations $\pi_{\tau}$ over partial flag varieties are $\Gamma$-equivariant, whenever it makes sense.

\begin{prop}
	Let ${\tau} \subset \Delta$ be a type of parabolic subgroups of $G_l$ such that $\Gamma \ast {\tau} = {\tau}$. Then, the orbit $X(\tau)$ is stable under the action of $\Gamma$ and the $(G_l \times G_l)$-equivariant morphism $$X(\tau) \to \Par_{\tau}(G_l) \times \Par_{{\tau}^{\opp}}(G_l)$$ mapping $e_{(T_l,B),\tau}$ to $(P_{\tau}, P_{\tau}^{\opp})$ is $\Gamma$-equivariant. In particular, it is defined over $k$.
\end{prop}

\begin{proof}
	Because the actions of $\Gamma$ and $G_l \times G_l$ are compatible, in the sense that the action map $G_l \times G_l \times \ol G_l \to \ol G_l$ is $\Gamma$-equivariant, it is sufficient to check that $\pi_{\tau}(\gamma(e_{(T_l,B),\tau})) = \gamma(\pi_{\tau}(e_{(T_l,B),\tau}))$. Now, by a previous computation, we have $\gamma(e_{(T_l,B),\tau}) = (\dot w, \dot w)^{-1}e_{(T_l,B),\tau}$, where $\dot w \in N_G(T)(l)$ is a lift of the element $w$ of the Weyl group such that $w(\gamma(\Delta)) = \Delta$. It follows that: $$\pi_{\tau}(\gamma(e_{(T_l,B),\tau})) = (\dot w, \dot w)^{-1}(P_{\tau}, P_{\tau}^{\opp})$$
	On the other hand, we have $$\gamma(P_{\tau}) = P_{G_l}(\gamma(\lambda_{\tau})) = P_{G_l}(\dot w^{-1}(\lambda_{\tau})) = \dot w^{-1} P_{G_l}(\lambda_{\tau}) \dot w = (\dot w, \dot w)^{-1} P_{\tau}$$
	and, likewise, $\gamma(P_{\tau}^{\opp}) = (\dot w, \dot w)^{-1}P_{\tau}^{\opp}$.
\end{proof}

Because the morphism $X(\tau) \to \Par_{\tau}(G) \times \Par_{{\tau}	^{\opp}}(G)$ is a locally trivial fibration with fibre a $k$-group, we have the following corollary, which can be thought of as a very special but explicit case of \cite[Theorem B]{Gabber}.

\begin{cor}\label{rat_pts}
	Let $\tau \subset \Delta$ be a type of parabolic subgroups of $G_l$ such that $\Gamma \ast \Delta = \Delta$. Then, the orbit $X({\tau})$ contains a $k$-rational point if and only if $\Par_{\tau}(G)(k) \ne \emptyset$.
	It follows that, if $G$ is anisotropic, then $\ol G \setminus G$ contains no rational points.
\end{cor}

From this corollary, we retrieve the following well-known result of Bruhat and Tits:

\begin{cor}[Bruhat-Tits-Rousseau]
	Assume that $k$ is locally compact and that $G$ is a $k$-anisotropic reductive group. Then, $G(k)$ is compact.
\end{cor}

\begin{proof}
	By the exactness of the sequence $ Z(G)(k) \to G(k) \to (G/Z(G))(k)$, it is sufficient to prove that $Z(G)(k)$ and $(G/Z(G))(k)$ are compact. It is therefore sufficient to treat the case when $G$ is a torus and the case when $G$ is semisimple adjoint. In both cases, we find a proper variety $X$ and an open immersion $G \hookrightarrow X$ such that $G(k) = X(k)$, which will conclude.
	
	The case of a torus is treated in the proof of \cite[Lemme 12]{CTS}.
	
	If $G$ is semisimple, adjoint, and anisotropic, then, by the previous corollary, the wonderful compactification morphism $G \hookrightarrow \ol G$ satisfies $G(k) = \ol G(k)$
\end{proof}

%% file: berko.tex
    \section{Compactifications of buildings}
    
   From this point on,  $G$ shall be a semisimple adjoint algebraic group over a non-archimedean field $k$ satisfying the conditions in \cite[1.3.4]{RTW1}. In this section, we summarise the construction of the family of Satake-Berkovich compactifications of the Bruhat-Tits building of $G$ defined in \cite{RTW1} and the wonderful compactifications defined in \cite{RTW3}, then extend the homeomorphism from \cite[Theorem 4.1]{RTW3} between the maximal Satake-Berkovich compactification and the wonderful compactifications to the case where $G$ is not split.  This will follow from the identification in \cite{RTW3} applied to $G_l$, and a generalisation of some of its results.
    
    \subsection{A primer of Berkovich geometry}
    
    We begin by recalling some definitions from Berkovich geometry that will be necessary to formulate and understand the rest of the text.
    
    We call non-archimedean field any field $k$ that is equipped with a nontrivial non-archimedean absolute value $|\cdot|$ with respect to which it is complete. We call $k$ local if it is locally compact.
    
    A non-archimedean extension $k'/k$ is an isometric extension of non-archimedean fields.
    
    If $(A, ||\cdot||)$ is a Banach algebra over $k$, we denote by $\rho: f \mapsto \lim ||f^n||^{1/n} = \inf ||f^n||^{1/n}$ its spectral seminorm. We then set $$A^{\oo} = \{x \in A, \rho(f) \le 1\}$$ its subalgebra of power-bounded elements and by $$A^{\oo\oo} = \{x \in A, \rho(f) < 1\}$$ the ideal of topologically nilpotent elements. We also denote by $\tilde A = A^{\oo}/A^{\oo\oo}$ the residue ring.
    
    In the following, we will mostly work with analytic spaces originating from varieties, therefore we will dispense with more general definitions. We denote by $k$ a non-archimedean field.
    \begin{itemize}
    \item If $X$ is an affine scheme over $k$, then we denote by $X^{\an}$ the set of non-archimedean multiplicative seminorms on $k[X]$. If $x \in X^{\an}$ and $f \in k[X]$, we will write $|f|(x)$ rather than $x(f)$. We endow $X^{\an}$ with the coarsest topology that makes the maps $x \mapsto |f|(x)$ continuous, for $f$ ranging in $k[X]$.
    
    This assignation is functorial and extends to a functor from the category of schemes over $k$ to that of analytic spaces over $k$ by glueing.
    
    \item Given $r_1, \dots, r_n > 0$, we set $$k\{r_1^{-1}x_1, \dots, r_n^{-1}x_n\} = \left\{\sum_{\nu \in \NN^n} a_{\nu}x^{\nu} \in k[[x_1, \dots, x_n]], |a_{\nu}|r_1^{\nu_1}\dots r_n^{\nu_n} \underset{n \to \infty}{\longrightarrow} 0\right\}$$ and endow it with the norm $\left|\sum_{\nu \in \NN^n} a_{\nu}x^{\nu} \right| =\max_{\nu \in \NN^n} |a_{\nu}|r_1^{\nu_1}\dots r_n^{\nu_n} $.
    
    \item A Banach algebra $(A, ||\cdot||)$ over $k$ is called $k$-affinoid if there exist $r_1, \dots, r_n > 0$ and a continuous epimorphism $$\phi: k\{r_1^{-1}x_1, \dots, r_n^{-1}x_n\} \epi A$$ such that the induced morphism $k\{\underline{r}^{-1}\underline{x}\}/\ker \phi \to A$ is an isometric isomorphism. It is called strictly $k$-affinoid if we may choose the $r_i$ in the value group $|k^{\times}|$.
    The space $\mc M(A)$ of multiplicative seminorms on $A$ that are bounded with respect to $||\cdot||$ is then called a $k$-affinoid domain. It sits naturally in $\AA^{n, \an}$ as a compact subspace.
    
    \item In \cite[1.2.4]{RTW1}, the authors define a construction allowing one to associate to any finitely presented $k^{\oo}$-scheme $\mc X$ a strictly $k$-affinoid domain $\mc X^{\an}$ inside the analytification of the generic fibre $(X \otimes_{k^{\oo}} k)^{\an}$.

    \item Given a commutative Banach $k$-algebra $A$, there exists a so-called reduction map $$\rho: \mc M(A) \to \Spec(\tilde A)$$ defined as follows: Given a point $x \in \mc M(A)$, denote by $\mf p_x = \{f \in A, |f|(x) = 0\}$ and denote by $\kappa(x)$ its completed residue field, ie. $$\kappa(x) = \widehat{\mathrm{Frac}(A/\mf p_x)}$$ with the completion taken with respect to the seminorm given by $x$, which becomes a multiplicative norm on $A/\mf p_x$. Then, the kernel of the canonical map $$A^{\oo} \to \kappa(x)$$ is a prime ideal containing $A^{\oo\oo}$, which therefore defines a point $\rho(x) \in \Spec(\tilde A)$.
    
    Then, if $A$ is strictly $k$-affinoid, $\rho$ is surjective and the inverse image of the set of generic points of $\Spec(\tilde A)$ is the so-called \textit{Shilov boundary} $\Gamma(A)$ of $\mc M(A)$, that is, it is a minimal closed subset of $\mc M(A)$ such that, for any $f \in A$, we have $$\max_{x \in \mc M(A)} |f|(x) = \max_{x \in \Gamma(A)} |f|(x).$$
    
    \end{itemize}
    
    \subsection{Satake-Berkovich compactifications}
    
    We now remind the reader of the results of \cite{RTW1} that we shall use in the following.
    
    We begin with partial functoriality of Bruhat-Tits buildings with respect to field extensions -- a property that will be ubiquitous in the following.
    
    Under the assumptions on the non-archimedean field $k$ brought forth in \cite[1.3.4]{RTW1}, there exist
    \begin{enumerate}
        \item A subcategory $\textbf{E}_0(k)$ of the category of non-archimedean extensions of $k$ that is cofinal, that is for any non-archimedean extension $k'/k$, there exists an extension $K/k'$ such that $K \in \textbf{E}_0(k)$ 
        \item A functor $\mc B(G, \cdot)$ on the category $\textbf{E}_0(k)$ that associates to each extension $K'/K$ in $\textbf{E}_0(k)$ a $G(K)$-equivariant injection $$\mc B(G, K) \hookrightarrow \mc B(G, K').$$
    \end{enumerate}
    Moreover, if $k$ is discretely valued with a perfect residue field (eg. if $k$ is local or $\CC((t))$), then we may take $\textbf{E}_0(k)$ to be the whole category of non-archimedean extensions of $k$.
    In the following, all non-archimedean extensions will be assumed to lie in $\textbf{E}_0(k)$.\\
    Finally, if $x \in \mc B(G, k)$ and $K \in \textbf{E}_0(k)$, then we denote by $x_K$ the image of $x$ in $\mc B(G, K)$ under the canonical embedding $$\mc B(G, k) \hookrightarrow \mc B(G, K)$$ associated to the inclusion $k \subset K$.
    
    Next, recall from \cite{RTW1} that there exists a canonical embedding $\vtheta: \mc B(G, k) \to G^{\an}$ such that, for each non-archimedean extension $k'/k$, the diagram 
    \begin{equation}\label{cdtheta}
    \begin{tikzcd}
    \mc B(G, k') \arrow[r, "\vtheta_{k'}"] & G_{k'}^{\an} \arrow[d, "{(pr_{k'/k})^{\an}}"] \\
    \mc B(G, k) \arrow[u] \arrow[r, "\vtheta"] & G^{\an}
    \end{tikzcd}
    \end{equation}
    commutes, where the leftmost vertical arrow comes from partial functoriality with respect to field extensions.
    
    The idea of the embedding can be understood from the following two observations:
    \begin{enumerate}
        \item If $G$ is split and $x$ is a special vertex in $\mc B(G, k)$, then the parahoric subgroup $\mf G_x^{\oo}$ associated to $x$ by Bruhat-Tits theory uniquely determines the point $x_K$ for each non-archimedean extensions $K/k$, in the sense that $\{x_K\} = \mc B(G, K)^{\mf G_x^{\oo}(K^{\oo})}$. 
        \item Any point in the building can be made into a special vertex in the building of $G_K$ for $K/k$ a suitable non-archimedean extension such that $G_K$ is split. \cite[Proposition 1.6]{RTW1}
    \end{enumerate} We then associate to each point $x$ of the building the strictly affinoid subgroup $$G_x = \pr_{K/k}^{\an}(\mf G_{x_K}^{\oo})^{\an}$$ of $G^{\an}$, where $\mf G_{x_K}^{\oo}$ is the connected parahoric subgroup associated to $x_K$. This analytic subgroup refines the parahoric group scheme associated to $x$. It is characterised by the property that, for each non-archimedean extension $K/k$, we have $$G_x(K) = \Stab_{G(K)}(x_K).$$ The advantage of this construction over the algebraic one is that the subgroup $G_x$ is uniquely associated to $x$, whereas the integral model of $G$ associated to $x$ by Bruhat-Tits theory depends only on the enclosure $\cl(x)$, that is the facet of the building that contains $x$. The subgroup $G_x$ in turn happens to be the holomorphically convex hull of a single point in $G^{\an}$, its \textit{Shilov boundary}, which we denote by $\vtheta(x)$.
    
    From this construction is then deduced a family of $G(k)$-equivariant maps $$(\vartheta_{t}: \mc B(G,k) \to \Par_{t}(G)^{\an})_{t}$$ indexed by $k$-rational types $t$ of parabolic $k$-subgroups of $G$. More precisely, for each parabolic $k$-subgroup $P$, we define a map $\lambda_P: G \to \Par(G)$ by mapping $g$ to $gPg^{-1}$. Then, the composition $\lambda_P^{\an}\circ \vartheta$ is easily proved \cite[Lemma 2.14]{RTW1} to be solely dependent on the type of $P$ and denoted by $\vartheta_{t(P)}$ . The results of \cite[section 3.4]{RTW1} then establish that, denoting the closure of their respective images by $\ol{\mc B}_{\tau}(G,k)$, the induced maps $$\vtheta_{t}: \mc B(G, k) \to \ol{\mc B}_{t}(G,k)$$ are continuous, open, and injective if $t$ is non-degenerate (if $G$ is almost simple, it is equivalent to it being the type of a proper parabolic subgroup).
    
    These constructions are referred to as \textit{Berkovich} or \textit{Satake-Berkovich} compactifications. 
    
    In the special case where $t = \emptyset$, the boundary admits the following description:
    
    \begin{thm}[{\cite[Theorem 4.1]{RTW1}}]\label{Boundary_RTW}
    	For each parabolic subgroup $P$ of $G$, there exists a canonical $P(k)$-equivariant embedding $\mc B(P_{\sss}, k) \hookrightarrow \Bor(G)^{\an}$, where $P_{\sss} = P/R(P)$ is the semisimple quotient of $P$.
    	Moreover, the buildings $\mc B(P_{\sss}, k)$, as $P$ runs over all parabolic subgroups of $G$, define a stratification of $\ol{\mc B}_{\emptyset}(G, k)$ into locally closed subspaces:
    	$$\ol{\mc B}_{\emptyset}(G,k) = \bigsqcup_{P \in \Par(G)(k)} \mc B(P_{\sss}, k).$$
    \end{thm}
    
    For an arbitrary type $t$, the description is analogous but the range of parabolic subgroups involved must be restricted to the $t$-relevant parabolic subgroups \cite[Definition 3.7]{RTW1}.
    
    \subsection{Wonderful compactifications}
    
    In addition to defining the embedding $\vtheta$, the article \cite{RTW1} also defines a "relative position" map $$\Theta: \mc B(G, k) \times \mc B(G, k) \to G^{\an}$$ using $\vtheta$. The idea is, as for $\vtheta$, to exploit the high transitivity properties and compatibility with field extensions earned by passing to Berkovich spaces. Specifically, crucial in both cases is the fact that, for any two points $x, y \in \mc B(G,k)$, there exist an affinoid field extension $K/k$ and an element $g \in G(K)$ such that $gx = y$. Here, it allows to set $\Theta(x,y) = \pr_{K/k}(g\vtheta_K(x))$, where $\vtheta_K: \mc B(G, K) \to G_K^{\an}$ is the embedding over $K$. The commutativity of the diagram \ref{cdtheta} from the previous section then allows one to prove the independence on $g$ and $K$.
    
    This map is extended in \cite[Section 2]{RTW3}, in the case where $G$ is split over $k$, to a map $$\ol{\Theta}: \mc B(G, k) \times \ol{\mc B}_{\emptyset}(G,k) \to \ol{G}^{\an}.$$
    The construction, which spans the entirety of the second section proceeds by first extending the map to partial compactifications of apartments as follows.
    Consider a Killing couple $(T, B)$ and $x$ a special point of the apartment $A = A(T, k)$ associated to $T$ in $\mc B(G,k)$. Denote by $\Phi^+ = \Phi(B,T)$ and $\Phi^- = -\Phi(B,T)$ the sets of positive and negative roots corresponding to $B$. The point $x$ defines a simplicial isomorphism between $A$ and $X_*(T)\otimes_{\ZZ}\RR$ or, equivalently, $\Hom_{\mrm{Grp}}(X^*(T), \RR_{>0})$.	
    One then sets $\ol A^B = \Hom_{\Mon}(\langle \Phi^- \rangle, \RR_{\ge 0})$, where $\langle \Phi^-\rangle$ is the monoid generated by $\Phi^-$ in $X^*(T)$, and defines the embedding $A \hookrightarrow \ol A^B$ by restriction. Geometrically, this amounts to compactifying only the Weyl cone $$\mf C(B) = \{x \in A(T, k)\, |\, \forall \alpha \in \Phi^+, \alpha(x) \ge 1\}$$ by allowing $\alpha(x)$ to assume infinite values for positive $\alpha$.
    
    The reason for introducing these partial compactifications is twofold. On the one hand, $\mc B(G, k) \times \ol{\mc B}_{\emptyset}(G, k)$ is covered by the products $A \times \ol A^B$ for all pairs $(T,B)$. On the other hand, the image of any point $(x, y) \in A\times \ol A^B$ lies in the analytification of an affine scheme, namely the partial compactification of the open Bruhat cell of the wonderful completion of $G$ \cite[Proposition 2.3]{RTW3}, which makes these compactifications amenable to explicit computations.
    
    The global construction is achieved by glueing the embeddings $\Theta_{(T,B)}: \mc A \times \ol A^B$ for all pairs $(T,B)$.
    
    The main result \cite[Theorem 4.1]{RTW3} of the article states that, if $G$ is split, then, for any $x \in \mc B(G, k)$, the map $\ol{\Theta}(x, \cdot):  \ol{\mc B}_{\emptyset}(G,k) \to \ol G^{\an}$ is a $G(k)$-equivariant homeomorphism onto its image for the $G(k)$-action on $\ol G^{\an}$ by left translation.
    
    In the process, \cite[Proposition 3.1]{RTW3} shows a remarkable compatibility between the boundary structures of $\ol{\mc B}_{\emptyset}(G, k)$ and that of $\ol G$. Namely:
    \begin{itemize}
    	\item For each proper parabolic subgroup $P$ of $G$ of type $\tau$, the map $\ol{\Theta}(x, \cdot)$ maps the boundary stratum $\mc B(P_{\sss}, k)$ to the boundary stratum $X(\tau)^{\an}$ in $\ol G^{\an}$.
    	\item Furthermore, letting $\pi_{\tau}^{\an}: X(\tau)^{\an} \to (\Par_{\tau}(G)\times \Par_{\tau^{\opp}}(G))^{\an}$ be the analytification of the fibration defined in 1.1, we have for each $y \in \mc B(P_{\sss}, k)$, $$\pi_{\tau}^{\an}(\ol{\Theta}(x,y)) = (P, \vtheta_{\tau}(x)).$$
    	In other words, the fibrations $\pi_{\tau}^{\an}$ separate the boundary components associated to parabolic subgroups of a single type $\tau$.
    \end{itemize}

    \subsection{Comparison of the compactifications in the general case}
    
    The extension of the above results to the case of an arbitrary reductive group follows from the split case, from functoriality of buildings with respect to field extensions, and from the descent of the fibrations constructed in section \ref{fibrations_dcp}. As in the previous case, the idea is to prove that each of the strata mentioned in Theorem \ref{Boundary_RTW} is mapped to a different fiber of the analytification of the fibration defined in \ref{classical_dcp}, and then prove that the restriction to each stratum is injective. The first part of this statement follows formally from the split case and the second part requires some explicit computations.
    
    \subsubsection*{Definition of the embedding}
    
    Our first task is to define the extension $$\ol{\Theta}: \mc B(G, k) \times \ol{\mc B}_{\emptyset}(G, k) \to \ol G^{\an}$$
    of the map $\Theta$ in our situation. Thankfully, the construction of this map, which occupies a large portion of \cite{RTW3} can be deduced from the split case as follows:
    Let $k'/k$ be a non-archimedean extension that splits $G$. Then, we let $\ol{\Theta}$ be the only map that makes the following diagram commute:
    \begin{center}
        \begin{tikzcd}
            \mc B(G, k') \times \ol{\mc B}_{\emptyset}(G, k') \arrow[r, "\ol{\Theta}"] & \ol G_{k'}^{\an} \arrow[d, "\pr_{k'/k}^{\an}"] \\
            \mc B(G, k) \times \ol{\mc B}_{\emptyset}(G,k) \arrow[u] \arrow[r, "\ol{\Theta}"] & \ol G^{\an}
        \end{tikzcd}.
    \end{center}
    The commutativity of the diagram 
    \begin{center}
        \begin{tikzcd}
            \mc B(G, k'') \times \mc B(G, k'') \arrow[r, "\Theta"] & G_{k''}^{\an} \arrow[d, "\pr_{k''/k'}^{\an}"] \\ 
            \mc B(G, k') \times \mc B(G, k') \arrow[u] \arrow[r, "\Theta"] & G_{k'}^{\an}
        \end{tikzcd}
    \end{center}
    and the density of $\mc B(G, k')$ in $\ol{\mc B}_{\emptyset}(G, k')$ for any tower of non-archimedean extensions $k''/k'/k$ then ensure that the above map $\ol{\Theta}$ is uniquely determined.
    
    \subsubsection*{Reduction to a stratum}
    
    We first establish that, for each $x$ in $\mc B(G, k)$, $\ol{\Theta}(x, \cdot)$ separates the boundary strata of $\ol{\mc B}_{\emptyset}(G, k)$, which will retrieve point 2 of the theorem cited in the introduction.
    
    \begin{prop}
    	Fix a proper parabolic subgroup $P$ of $G$ of type $\tau$.
    	\begin{enumerate}
    		\item The map $\ol{\Theta}$ maps $\mc B(G, k) \times \mc B(P_{\sss}, k)$ to the boundary stratum $X(\tau)^{\an}$ of $\ol G^{\an}$.
    		\item For all $(x,y) \in \mc B(G, k) \times \mc B(P_{\sss}, k)$, we have 
    		$$\pi_{\tau}^{\an}(\ol{\Theta}(x,y)) = (P, \vtheta_{\tau}(x)).$$
    	\end{enumerate}
    \end{prop}
    
    \begin{proof}
    	Let $l/k$ be a finite Galois extension such that $G_l$ is split.
    	\begin{enumerate}
    		\item Let $x \in \mc B(G, k)$.
    		By \cite[Proposition 4.5]{RTW1}, we have a commutative diagram:
    		\begin{equation}
    		\begin{tikzcd}\label{cd1}
    		\mc B(P_{\sss}, l) \arrow[r] & \ol{\mc B}_{\emptyset}(G, l) \arrow[r, "{\ol{\Theta}(x, \cdot)}"] & \ol{G_l}^{\an} \arrow[d, "(\pr_{l/k})^{\an}"]\\
    		\mc B(P_{\sss}, k)\arrow[u] \arrow[r] & \ol{\mc B}_{\emptyset}(G, k) \arrow[u] \arrow[r, "{\ol{\Theta}(x, \cdot)}"] & \ol{G}^{\an}.
    		\end{tikzcd}
    		\end{equation}
    		By \cite[Proposition 3.1]{RTW3}, the building $\mc B(P_{\sss}, l)$ is mapped to the analytified orbit $X(\tau)_{l}^{\an}$. The commutativity of the following diagram
    		\begin{equation}
    		\begin{tikzcd}
    		X(\tau)_l^{\an} \arrow[d, "{\pr_{l/k}^{\an}}"] \arrow[r, "\subset"] & \ol{G_l}^{\an}\arrow[d, "{\pr_{l/k}^{\an}}"] \\
    		X(\tau)^{\an} \arrow[r, "\subset"] & \ol G^{\an}
    		\end{tikzcd}
    		\end{equation}
    		then ensures that $\mc B(P_{\sss}, k)$ is also mapped to $X(\tau)^{\an}$. 
    		\item This follows from the commutativity of the diagram
    		\begin{equation}\label{cd}
    		\begin{tikzcd}
    		\mc B(P_{\sss}, l) \arrow[r, "{\ol{\Theta}(x, \cdot)}"] & X(\tau)_l^{\an} \arrow[r, "\pi_{\tau}^{\an}"] \arrow[d, "\pr_{l/k}^{\an}"] & (\Par_{\tau}(G_l) \times \Par_{\tau^{\opp}}(G_l))^{\an} \arrow[d, "\pr_{l/k}^{\an}"] \\
    		\mc B(P_{\sss}, k) \arrow[u] \arrow[r, "{\ol{\Theta}(x, \cdot)}"] & X(\tau)^{\an} \arrow[r, "\pi_{\tau}^{\an}"] & (\Par_{\tau}(G) \times \Par_{\tau^{\opp}}(G))^{\an}
    		\end{tikzcd}
    		\end{equation}	
    		and \cite[Prop 3.1, (ii)]{RTW3}.
    	\end{enumerate}
    	
    \end{proof}

    \subsubsection*{A formula for {$\Theta$}}
    
    To finish proving that each $\ol{\Theta}(x, \cdot)$ is injective, it is therefore sufficient to prove that each of them is injective when restricted to each boundary stratum $\mc B(P_{\sss}, k)$. By definition, we have the following commuting square for each proper parabolic subgroup $P$ of $G$ and any non-archimedean extension $k'/k$ that splits $G$: 
    \begin{equation}
    \begin{tikzcd}
    \mc B(P_{\sss}, k') \arrow[r, "{\ol{\Theta}(x_{k'}, \cdot)}"] & \ol{G_{k'}}^{\an} \arrow[d, "(\pr_{{k'}/k})^{\an}"]\\
    \mc B(P_{\sss}, k)\arrow[u] \arrow[r, "{\ol{\Theta}(x, \cdot)}"] & \ol{G}^{\an}.
    \end{tikzcd}
    \end{equation}
    The leftmost arrow is injective by functoriality of buildings with respect to field extensions and the topmost one is injective by \cite[Proposition 3.1]{RTW3}. It is therefore sufficient to prove the injectivity of the map $\pr_{{k'}/k}^{\an}$ restricted to $\mc B(P_{\sss}, k)$. 
    
    To this end, we work locally which, in the setting of Berkovich spaces, amounts to working with seminorms on a certain $k$-algebra. Specifically, we seek precise formulas akin to \cite[Lemma 2.2]{RTW3} for the image of a pair of points in an apartment involving the relative root system.
    
    The first step towards that goal is to adapt \cite[Prop 2.16]{RTW1} in order for it to apply in the case of a non-necessarily split group.
    
    Let $P$ be a minimal parabolic subgroup of $G$, $S$ be a maximal split torus in $G$, $L = Z_G(S)$ be the Levi factor of $P$ corresponding to $S$, and $o$ be a special vertex in $A(S, k)$. Denote by $_k \Phi$ the relative root system $\Phi(G, S)$, by $_k \Delta$ the basis of the root system corresponding to $P$. Following \cite[2.1.2]{RTW1}, we prove that $\vtheta(o) \in G^{\an}$ lies in the analytification of the open Bruhat cell associated to the pair $(S, P)$ $$\Omega(S, P) = R_u(P^{\opp}) \times L \times R_u(P)$$ and express it as a seminorm on $k[\Omega(S, P)]$.
    
    We introduce a non-archimedean extension $K/k$ such that $G_K$ is split and $o_K$ is a special vertex in $\mc B(G, K)$ (which, despite $o$ being special, is no guarantee). The point being that $\vtheta(o) = \pr_{K/k}^{\an}(\vtheta(o_K))$ and that $\vtheta(o_K)$ is easily described in terms of our data.
    Let $T \subset P_K$ be a maximal split torus containing $S_K$ and $B \subset P_K$ be a Borel subgroup containing $T$. Consider $\Phi = \Phi(G_K, T)$ the absolute root system and $\Phi^{\pm} = \pm\Phi(B, T)$ the system of positive roots corresponding to $B$. Finally, let $\Phi_L = \Phi(L_K, T)$ be the absolute root system of $L$.
    
    Recall from \cite[5.1.3]{BT2} that there exists a natural $L(K)$-equivariant embedding $$\mc B^e(L, K) \hookrightarrow \mc B(G, K)$$ of the extended Bruhat-Tits building of the Levi factor $L_K$ into the building $\mc B(G, K)$ of $G_K$ as the union of the apartments $A(T_1, K)$ for $T_1$ ranging among the maximal split tori of $G_K$ contained in $L_K$. Then, identifying $\mc B^e(L, K)$ with its image, we observe that $o_K$ lies in $\mc B^e(L,K)$ and therefore defines a $K^{\oo}$-model $\mc L$ of $L_K$, its associated parahoric group scheme ($\mf L_o^{\oo}$ with the notations of \cite{BT2}). Set $\mc S = Z(\mc L)$.

    Moreover, $o_K$ determines an integral model $\mc G$ of $G_K$ as well as a collection of $K^{\oo}$-isomorphisms $$(u_{\alpha}: \AA^1_{K^{\oo}} \to \mc U_{\alpha})_{\alpha \in \Phi}$$ onto $K^{\oo}$-models of the root groups $U_{\alpha}$ of $G_K$. These yield an open immersion of $K^{\oo}$-schemes $$\prod_{\alpha \in \Phi^-\setminus \Phi_L}\AA^1_{K^{\oo}} \times \mc L \times \prod_{\alpha \in \Phi^+\setminus \Phi_L}\AA^1_{K^{\oo}} \hookrightarrow \mc G.$$
    We denote its image by $\Omega(\mc S, \mc P)$.

    We may now state our first proposition:
    
    \begin{prop} The image $\vtheta(o_K)$ of $o_K$ in $G_K^{\an}$ belongs to $\Omega(S, P)_K^{\an}$ and corresponds to the following norm on $K[\Omega(S, P)] =  K[L][(\xi_{\alpha})_{\alpha \in \Phi \setminus \Phi_L}]$:
    	$$f = \sum_{\nu \in \NN^{\Phi \setminus \Phi_L}} f_{\nu} \xi^{\nu} \mapsto \max_{\nu \in \NN^I} |f_{\nu}|(\vtheta_L(o_K)),$$
    	where $\vtheta_L(o_K)$ is the image of the canonical embedding associated to $L$ $$\vtheta_L: \mc B(L, K) \to L_K^{\an}.$$
    \end{prop}
    
    \begin{proof}
    	The proof is analogous to that of \cite[Prop 2.6]{RTW1}.
    	
    	Recall that, because $G_K$ is split, $\vtheta(o_K)$ is the only point of the Shilov boundary of the affinoid subgroup $\mc G^{\an}$ of $G_K^{\an}$, that is, by \cite[1.2.4]{RTW1}, the only point in $\mc G^{\an}$ that is mapped to the generic point of $\mc G_{\tilde K}$ by the reduction map $$\rho: \mc G^{\an} \to \mc G_{\tilde K}.$$
    	Because $\Omega(\mc S, \mc P)$ is an open subset of $\mc G$ that meets the special fibre, its special fibre $\Omega(\mc S, \mc P)_{\tilde k}$ contains the generic point of $\mc G$. As a consequence of the commutativity of the following diagram
    	\begin{equation*}
    	\begin{tikzcd}
    	\Omega(\mc S, \mc P)^{\an} \arrow[d, "\subset"] \arrow[r, "\rho"] & \Omega(\mc S, \mc P)_{\tilde K} \arrow[d, "\subset"] \\ \mc G^{\an} \arrow[r, "\rho"] & \mc G_{\tilde K}
    	\end{tikzcd},
    	\end{equation*}
    	the open subset $\Omega(\mc S, \mc P)^{\an}$ of $\mc G^{\an}$ must contain $\vtheta(o_K)$.
    	Moreover, the condition that $\vtheta(o_K)$ reduces to the generic point of $\Omega(\mc S, \mc P)_{\tilde K}$ may be rewritten as the claim that, for any function $f \in k[\Omega(S, P)] = K[L][(\xi_{\alpha})_{\alpha \in \Phi \setminus \Phi_L}]$, we have the equivalences:
    	$$|f|(\vtheta(o_K)) \le 1 \Longleftrightarrow f \in K^{\oo}[L][(\xi_{\alpha})_{\alpha \in \Phi \setminus \Phi_L}] $$ and $$|f|(\vtheta(o_K)) < 1 \Longleftrightarrow f \text{ maps to zero in } (\tilde K\otimes_{K^{\oo}} K^{\oo}[\mc L])[(\xi_{\alpha})_{\alpha \in \Phi \setminus \Phi_L}].$$
    	It follows that the restriction of $\vtheta(o_K)$ to $K[L]$ is the Shilov boundary of the affinoid subgroup $\mc L^{\an} \subset L_K^{\an}$, which is precisely the image of $o_K$ in $L^{\an}$ via $\vtheta_L(o_K)$. 
    	Furthermore, if $f = \sum_{\nu \in \NN^{\Phi\setminus \Phi_L}} f_{\nu}\xi^{\nu}$, we have $$|f|(\vtheta(o_K)) = \max_{\nu\in \NN^I}|f_{\nu}|(\vtheta_L(o_K)).$$
    \end{proof}

    Having established the above formula for $\vtheta(o_K)$, we are now able to describe the norm $\Theta(x_K,y_K)$ associated to a pair of points in $A(S, k)$. We preserve the above notation.
    
    \begin{prop}\label{Formule_Theta}
    	Let $x$ and $y$ be points in $A(S, k)$. Then, the point $\Theta(x_K,y_K)$ lies in the analytification $\Omega(S, P)_K^{\an}$ of the open Bruhat cell of $G_K$ corresponding to the pair $(S_K, P_K)$.
    	Moreover, identifying the apartment $A(T,K)$ with $\Hom(X^*(T), \RR_{> 0})$ using the special point $o_K$ as origin, we have for any $f = \sum_{\nu \in \NN^{\Phi \setminus \Phi_L}} f_{\nu} \xi^{\nu} \in K[\Omega(S, P)] = K[L][(\xi_{\alpha})_{\alpha \in \Phi\setminus\Phi_L}]$: $$|f|(\Theta(x_K,y_K)) = \max_{\alpha \in \Phi\setminus \Phi_L}  |f_{\nu}|(\vtheta_{L}(o_K)) \prod_{\alpha \in \Phi^-\setminus \Phi_L} \langle \alpha, y\rangle^{\nu_{\alpha}} \prod_{\alpha \in \Phi^+\setminus \Phi_L} \langle \alpha, x\rangle^{\nu_{\alpha}}$$
    \end{prop}
    
    \begin{proof}
    	By \cite[Prop 1.7]{RTW1}, there exist a non-archimedean extension $K'/K$ and elements $s, t \in T(K')$ such that $x_{K'} = so_{K'}$ and $y_{K'} = to_{K'}$. Hence, by \cite[Prop 2.12]{RTW1}, we have: $$\Theta(x_{K'}, y_{K'}) = t\Theta(o_{K'}, o_{K'})s^{-1} = t\vtheta(o_{K'})s^{-1}.$$
    	Consequently, for any $f = \sum_{\nu \in \NN^{\Phi \setminus \Phi_L}} f_{\nu} \xi^{\nu} \in K[\Omega(S, P)]$, we have
    	\[
    	|f|(\Theta(x_K,y_K)) = \left|\sum_{\nu \in \NN^{\Phi\setminus\Phi_L}} f_{\nu}\prod_{\alpha \in \Phi^-\setminus \Phi_L} \alpha(t)^{\nu_{\alpha}} \prod_{\alpha \in \Phi^+\setminus \Phi_L} a(s)^{\nu_{\alpha}} \xi^{\nu}\right|(\vtheta(o_{K'}))
    	\]
    	and therefore, by the previous proposition
    	\[
    	|f|(\Theta(x_K,y_K)) = \max_{\nu \in \NN^{\Phi\setminus\Phi_L}} \left|f_{\nu}\prod_{\alpha \in \Phi^- \setminus \Phi_L} \alpha(t)^{\nu_{\alpha}} \prod_{\alpha \setminus \Phi^+\setminus \Phi_L} \alpha(s)^{\nu_{\alpha}} \right|(\vtheta_L(o_K))
    	\]
    	and finally:
        \[
        |f|(\Theta(x_K, y_K)) = \max_{\nu \NN^{\Phi\setminus \Phi_L}}|f_{\nu}|(\vtheta_L(o_K))\prod_{\alpha \in \Phi^- \setminus \Phi_L} |\alpha(t)|^{\nu_{\alpha}} \prod_{\alpha \setminus \Phi^+\setminus \Phi_L} |\alpha(s)|^{\nu_{\alpha}}
        \]
    	which concludes as, by the identification $A(T, K) \simeq \Hom(X^*(T), \RR_{>0})$ given by $o_K$, we have $$|\chi(s)| = \langle \chi, x\rangle \text{ and } |\chi(t)| = \langle \chi, y \rangle$$ for any $\chi \in X^*(T)$.
    \end{proof}
    
    \subsubsection*{Injectivity of the projection}
    
    We now return to the initial problem of proving the injectivity of $\ol{\Theta}(x, \cdot)$ for any $x \in \mc B(G, k)$. As we previously established, it reduces to proving the injectivity of $(\pr_{K/k})^{\an}: \ol{G_K}^{\an} \to \ol G^{\an}$ in restriction to boundary strata $\mc B(Q_{\sss}, k)$.
    
    \begin{prop}
    Let $Q$ be a proper parabolic subgroup and $x \in \mc B(G,k)$. The restriction of the projection $(\pr_{K/k})^{\an}$ to the image of $\mc B(Q_{\sss}, k)$ under $\ol{\Theta}(x_{K}, \cdot)$ is injective.
    \end{prop}
    
    \begin{proof}
   Fix $x \in \mc B(G, k)$ and let $y$ and $y'$ be two points in $\mc B(Q_{\sss}, k)$. There exists a maximal split torus $S \subset Q$ such that $y$ and $y'$ lie in $\ol{A(S, k)}$. Indeed, there exists a maximal split torus $\ol S \subset Q_{\sss}$ such that $y, y' \in A(\ol S, k)$. That torus is the image under the projection $Q \to Q_{\sss}$ of a unique maximal split torus $S$ of $Q$. Moreover, the inclusion $\mc B(Q_{\sss}, k) \hookrightarrow \ol{\mc B}_{\emptyset}(G, k)$ is constructed in such a way that $A(\ol S, k)$ lies in the boundary of $\ol{A(S, k)}$ \cite[Lemma 4.3]{RTW1}.
    In addition, by construction of the building, there exists $g \in G(k)$ such that $gx \in A(S, k)$.
    
     We therefore assume, without loss of generality, that $y$ and $y'$ both lie in the closure of an apartment $A(S,k)$ that contains $x$, replacing $y$ and $y'$ by $g^{-1}y$ and $g^{-1}y'$ if necessary. Fix a special point $o$ in $A(S, k)$.
    
    Let $P$ be a minimal parabolic subgroup of $Q$ containing $S$ and let $T$ be a maximal torus of $P$. Let $K/k$ be an affinoid extension such that $T_K$ is split and such that $o_K$ is a special vertex. Let $B$ be a Borel subgroup of $P_K$ containing $T_K$.
    
    Then, the images $y_K$ and $y'_K$ of $y$ and $y'$ in the Satake-Berkovich embedding $\ol{\mc B}_{\emptyset}(G, K)$ of the building $\mc B(G, K)$ lie in the partial compactification $\ol{A(T, K)}^{B}$ \cite[Section 2]{RTW3} of the apartment associated to the Weyl cone corresponding to $B$.
    
    Denote by $_k \Phi = \Phi(G, S)$ be the relative root system of $G$ and by $\Phi = \Phi(G_K, T_K)$ its absolute root system. Denote by $_k \Phi^+$ the set of positive relative roots corresponding to $P$ and by $\Phi^+$ the set of positive roots corresponding to $B$. Due to the relative positions of $S, T, P$ and $B$, the restriction to $S_k$ of a root in $\Phi^+$ is a root in $_k\Phi^+$.
    
    The point $o_K$ yields an isomorphism $$A(T, K) \simeq \Hom(X^*(T), \RR_{>0})$$
    as well as a compatible identification $$\ol{A(T, K)}^{B} \simeq \Hom_{\Mon}(\langle -\Phi^+ \rangle, \RR_{\ge 0}).$$
    
    Likewise, we have an isomorphism $$A(S, k) \simeq \Hom(X^*(S), \RR_{>0}).$$
    
    Note that, with these identifications, we have for any point $p \in A(S,k)$ and any character $\chi \in X^*(T)$ the identity: $$\langle \chi, p_K \rangle = \underbrace{\langle \chi_{|S}, p \rangle}_{\substack{\text{duality pairing between}\\X^*(S) \text{ and } A(S, k)}}.$$
    
    Assume that $y \ne y'$ and let $(y_n)$ and $(y'_n)$ be sequences in $A(S, k)$ converging to $y$ and $y'$ respectively.
    
    Then, by the above identification, there must exist a root $\alpha \in - \Phi^+\cap  \Phi(Z_Q(S)_K, T_K) $ such that $\langle \alpha, y\rangle \ne \langle \alpha, y'\rangle$, say, up to exchanging the roles of $y$ and $y'$, $\langle \alpha, y \rangle < \langle \alpha, y'\rangle$. We choose $\alpha \in \Phi(Z_Q(S)_K, T_K)$ to ensure that $\langle \alpha, y\rangle$ and $\langle \alpha, y'\rangle$ are both positive quantities. Then, there exists $\varepsilon > 0$ such that, for each $n \in \NN$ sufficiently large, we have $\langle \alpha, (y_n)_K \rangle > 0$ and $$\frac{\langle \alpha, (y'_n)_K\rangle}{\langle \alpha, (y_n)_K\rangle}\ge 1+\varepsilon,$$
    which, by the above remark, rewrites as $$\frac{\langle \alpha_{|S}, y'_n\rangle}{\langle \alpha_{|S}, y_n\rangle} \ge 1+\varepsilon.$$
    
    Set $a = \alpha_{|S}$ be the relative root corresponding to $\alpha$. Let $\phi \in k[\Omega(S,P)]$ be a nonzero linear form on $U_a$. Because we have an $S_K$-equivariant (and in particular $K$-linear) isomorphism $$(U_a)_K \simeq \prod_{\substack{\alpha \in \Phi\\ \alpha_{|S} = a}} U_{\alpha},$$ we may write $\phi = \sum_{\alpha_{|S} = a} \phi_{\alpha} \alpha$, with $\phi_{\alpha} \in K$ and not all $\phi_{\alpha}$ equal to zero.
    Then, Proposition \ref{Formule_Theta} yields $$|\phi|(\Theta(x, y_n)) = \left(\max_{\substack{\alpha \in \Phi\\ \alpha_{|S} = a}} |\phi_{\alpha}|\right) \langle a, y_n\rangle$$ and $$|\phi|(\Theta(x,y'_n)) = \left(\max_{\substack{\alpha \in \Phi\\ \alpha_{|S} = a}} |\phi_{\alpha}|\right)  \langle a, y'_n\rangle.$$
    Because $|\phi|(\cdot)$ is continuous on $\Omega(S, P)^{\an}$ and the above two sequences do not have the same limit, it follows that the sequences $\Theta(x, y_n)$ and $\Theta(x, y'_n)$ cannot have the same limit in any space that contains $\Omega(S, P)^{\an}$ as a topological subspace, and in particular, not in $\ol{G}^{\an}$. In particular, we must have $\ol{\Theta}(x, y) \ne \ol{\Theta}(x, y')$.

\end{proof}

Putting all of the above facts together, we finally get point 1 of the theorem cited in the introduction:

\begin{cor}
    For each $x \in \mc B(G,k)$, the map $$\ol{\Theta}(x, \cdot): \ol{\mc B}_{\emptyset}(G, k) \to \ol{G}^{\an}$$ is a $G(k)$-equivariant embedding.
\end{cor}

\begin{proof}
    Let $y, y' \in \ol{\mc B}_{\emptyset}(G, k)$. There exist parabolic subgroups $P$ and $Q$ of $G$ such that $y \in \mc B(P_{\sss}, k)$ and $y' \in \mc B(Q_{\sss}, k)$. Assume that $\ol{\Theta}(x, y) = \ol{\Theta}(x, y')$. Then, by Proposition 10, we must have $t(P) = t(Q)$ and $\pi_{t(P)}^{\an}(\ol{\Theta}(x,y)) = \pi_{t(P)}^{\an}(\ol{\Theta}(x,y')) $ and therefore $P = Q$. We conclude using Proposition 13.
\end{proof}

Moreover, in the case of a local field, this theorem does yield an identification between compactifications (point 3 in the theorem cited in the introduction).

\begin{cor}
    If $k$ is locally compact, the embedding $\ol{\Theta}(x, \cdot)$ is a homeomorphism onto the closure of the image of $\mc B(G, k)$ in $\ol G^{\an}$ under $\ol{\Theta}(x, \cdot)$.
\end{cor}

\begin{proof}
    Assume that $k$ is locally compact. By the above corollary, the map $$\ol{\Theta}(x, \cdot): \ol{\mc B}_{\emptyset}(G, k) \to \ol{G}^{\an}$$ is an embedding and, because $k$ is locally compact, $\ol{\mc B}_{\emptyset}(G, k)$ is compact \cite[Proposition 3.34]{RTW1}. Hence $\ol{\Theta}(x, \cdot)$ is closed and therefore a homeomorphism onto its image. Because $\mc B(G, k)$ is dense in $\ol{\mc B}_{\emptyset}(G, k)$, the image of $\ol{\Theta}(x, \cdot)$ is the closure of the image of $\mc B(G, k)$.
\end{proof}

%% file: fixed_points.tex
\section{Galois-fixed points in the compactification}

Let $G$ be a semisimple group over a local field $k$. In this section, given a Galois extension $k'/k$, we examine the fixed points of $\ol{\mc B}_{\emptyset}(G, k')$ under the action of $\Gal(k'/k)$. We check that the fixed points are precisely the limits of the fixed points in the building. In particular, in the case of a wildly ramified extension $k'/k$, the fixed locus is larger than $\ol{\mc B}_{\emptyset}(G, k)$ and contains undesirable components, the size of which is controlled by the wildness of the extension as follows from a result of Rousseau.

We begin with a refresher on fixed points in Bruhat-Tits theory.

\subsection{Fixed points in the building}

\subsubsection*{Unramified descent}

In the case of an unramified extension $k'/k$, the fixed point set $\mc B(G, k')^{\Gal(k'/k)}$ is an affine building and $\iota_{k,k'}$ identifies it with the Bruhat-Tits building of $G$ \cite[Section 3]{PrasadUnr}. Moreover, the simplicial structure of $\mc B(G,k)$ is very easy to relate to that of $\mc B(G,k')$: The simplices of $\mc B(G,k)$ are precisely the intersections of simplices of $\mc B(G,k')$ with the fixed-point set.\cite[3.2]{PrasadUnr}

\subsubsection*{Ramification and barbs}

In the tamely ramified case, the fixed point set is still a building and identified to the Bruhat-Tits building of $G$ (see \cite{Rou} and \cite{PrasadTR}). However, the simplicial structure of $\mc B(G, k)$ does not relate to that of $\mc B(G, k')$ as neatly as in the unramified case. 

In the general case, however, the fixed point set might not be a building. It contains the building of $G$ over $k$ but may contain some extra fixed points ("barbs"). The size of these extraneous components is bounded by the wildness of the extension. Recall that, if $k'/k$ is a finite Galois extension with residue characteristic $p$, then its wildness $s(k'/k)$ is the $p$-adic valuation of its ramification index. Then, we have the following statement:

\begin{thm}[{\cite[Prop 5.2.7]{Rou}}]
	There exists a constant $C$ depending only on the absolute rank of $G$ such that, for each non archimedean Galois extension $k'/k$, and for each $x \in \mc B(G, l)^{\Gal(k'/k)}$, we have $$d(x, \mc B(G, k)) \le C s(l/k).$$
\end{thm}

\begin{ex}
	Consider the case when $k = \QQ_p$, $k' = \QQ_p(\sqrt p)$, and $G = \SL_{3,k}$. Let $\varpi = \sqrt p$. Then, the Galois group of the extension $k'/k$ is generated by the involution $\sigma: a + b\varpi \mapsto a - b \varpi$. 
	
	Note that $\mc O_k = \ZZ_p$ and $\mc O_{k'} = \ZZ_p[\varpi]$.
	
	Recall from eg. \cite[19.1]{Garrett} that, in this case, the Bruhat-Tits building of $G$ may be described as a simplicial complex of dimension 2 whose vertices are homothety classes of $\mc O_k
	$-submodules of $k^3$ of maximal rank and where two vertices $[M_1]$ and $[M_2]$ are adjacent if and only if $M_1$ and $M_2$ may be chosen in such a way that $pM_1 \subsetneq M_2 \subsetneq M_1$.
	Then, the inclusion map $\iota_{k,k'}: \mc B(G, k) \to \mc B(G,k')$ may be described at the level of vertices as mapping the homothety class of an $\mc O_k$-lattice $M$ of $k^3$ to the homothety class of the $\mc O_{k'}$-lattice $M \otimes_{\mc O_k} \mc O_{k'}$ in $l^3$. Note that, because $k'/k$ is ramified, this map does not preserve adjacency. 
	
	Consider $e_1 = \begin{pmatrix}1\\0\\0\end{pmatrix}$, $e_2 = \begin{pmatrix}0\\1\\0\end{pmatrix}$ and $e_3 = \begin{pmatrix}0\\0\\1\end{pmatrix}$. Then, the subgroup $\begin{pmatrix}1 & 0 & 0 \\ 0 & 1 & 0\\ \varpi \mc O_{k'} & 0 & 1 \end{pmatrix}$ acts transitively on the alcoves incident to the edge
	\begin{center}
    \begin{tikzpicture}
	    \draw (0,0) node[below left]{$\left[\mc O_{k'} e_1 \oplus \mc O_{k'} e_2 \oplus \mc O_{k'} \varpi e_3\right ]$} node{$\bullet$};
        \draw (5,0) node[below right]{$\left[\mc O_{k'} e_1 \oplus \mc O_{k'} \varpi e_2 \oplus \mc O_{k'} \varpi e_3\right ]$} node{$\bullet$};
        \draw (0,0)--(5,0);
	\end{tikzpicture}
	\end{center}
	aside from  \begin{center}
    \begin{tikzpicture}
        \draw (0,0) node[below left]{$\left[\mc O_{k'} e_1 \oplus \mc O_{k'} e_2 \oplus \mc O_{k'} \varpi e_3\right ]$} node{$\bullet$};
        \draw (5,0) node[below right]{$\left[\mc O_{k'} e_1 \oplus \mc O_{k'} \varpi e_2 \oplus \mc O_{k'} \varpi e_3\right ]$} node{$\bullet$};
        \draw (2.5, 3) node[above]{$\left[\mc O_{k'} e_1 \oplus \mc O_{k'} e_2 \oplus \mc O_{k'}  e_3\right ]$} node{$\bullet$};
        \draw (0,0)--(5,0);
        \draw (0,0)--(2.5,3);
        \draw (2.5, 3)--(5,0);
    \end{tikzpicture}
    \end{center} An alcove containing these two vertices is completely determined by its third vertex, hence the alcoves incident to the edge \begin{center}
	\begin{tikzpicture}
	    \draw (0,0) node[below left]{$\left[\mc O_{k'} e_1 \oplus \mc O_{k'} e_2 \oplus \mc O_{k'} \varpi e_3\right ]$} node{$\bullet$};
        \draw (5,0) node[below right]{$\left[\mc O_{k'} e_1 \oplus \mc O_{k'} \varpi e_2 \oplus \mc O_{k'} \varpi e_3\right ]$} node{$\bullet$};
        \draw (0,0)--(5,0);
	\end{tikzpicture}
	\end{center} are precisely \begin{center}
    \begin{tikzpicture}
        \draw (0,0) node[below left]{$\left[\mc O_{k'} e_1 \oplus \mc O_{k'} e_2 \oplus \mc O_{k'} \varpi e_3\right ]$} node{$\bullet$};
        \draw (5,0) node[below right]{$\left[\mc O_{k'} e_1 \oplus \mc O_{k'} \varpi e_2 \oplus \mc O_{k'} \varpi e_3\right ]$} node{$\bullet$};
        \draw (2.5, 3) node[above]{$\left[\mc O_{k'} e_1 \oplus \mc O_{k'} e_2 \oplus \mc O_{k'}  e_3\right ]$} node{$\bullet$};
        \draw (0,0)--(5,0);
        \draw (0,0)--(2.5,3);
        \draw (2.5, 3)--(5,0);
    \end{tikzpicture}
    \end{center} and the
    \begin{center}
    \begin{tikzpicture}
        \draw (0,0) node[below left]{$\left[\mc O_{k'} e_1 \oplus \mc O_{k'} e_2 \oplus \mc O_{k'} \varpi e_3\right ]$} node{$\bullet$};
        \draw (5,0) node[below right]{$\left[\mc O_{k'} e_1 \oplus \mc O_{k'} \varpi e_2 \oplus \mc O_{k'} \varpi e_3\right ]$} node{$\bullet$};
        \draw (2.5, 3) node[above]{$\left [\mc O_{k'} (\varpi^{-1}e_1 + te_3)  \oplus \mc O_{k'} e_2 \oplus \mc O_{k'} \varpi e_3 \right ]$} node{$\bullet$};
        \draw (0,0)--(5,0);
        \draw (0,0)--(2.5,3);
        \draw (2.5, 3)--(5,0);
    \end{tikzpicture}
    \end{center}
    for $t \in \mc O_{k'}^{\times}$. Because $e_1$ and $e_2$ and $e_3$ are in $k^3$, the Galois involution fixes $[\mc O_{k'} e_1 \oplus \mc O_{k'} e_2 \oplus \mc O_{k'} e_3]$  and  $[\mc O_{k'} \varpi^{-1} e_1  \oplus \mc O_{k'} e_2 \oplus \mc O_{k'} \varpi e_3]$. Moreover, the quick computation below proves that there are no other fixed vertices if $p \ne 2$, but that all the above vertices are fixed if $p=2$.
	
	Indeed, the vertex $ [\mc O_{k'} (\varpi^{-1}e_1 + te_3)  \oplus \mc O_{k'} e_2 \oplus \mc O_{k'} e_3]$ is fixed if and only if $$-\varpi^{-1} e_1 + \sigma(t)e_3 = \alpha(\varpi^{-1} e_1 + e_3) + \beta e_2 + \gamma \varpi e_3$$ for some $\alpha, \beta, \gamma \in l$. We must have $\beta = 0$ and $\alpha = - 1$ and therefore $$t+\sigma(t) = \gamma \varpi.$$
	In other words, the point $[\mc O_{k'} e_1 \oplus \mc O_{k'} (\varpi^{-1}t e_1 + e_2)]$ is fixed if and only if $t + \sigma(t)$ is not invertible. If $t = a+\varpi b \in \mc O_{k'}^{\times}$ with $a,b \in \mc O_k$, then $a \in \mc O_k^{\times}$ and $t+\sigma(t) = 2a $ is invertible if and only if $2$ is invertible in $\mc O_k$, that is if and only if $p \ne 2$.
	
	If $p = 2$, then the alcove 
	\begin{center}
    \begin{tikzpicture}
        \draw (0,0) node[below left]{$\left[\mc O_{k'} e_1 \oplus \mc O_{k'} e_2 \oplus \mc O_{k'} \varpi e_3\right ]$} node{$\bullet$};
        \draw (5,0) node[below right]{$\left[\mc O_{k'} e_1 \oplus \mc O_{k'} \varpi e_2 \oplus \mc O_{k'} \varpi e_3\right ]$} node{$\bullet$};
        \draw (2.5, 3) node[above]{$\left [\mc O_{k'} (\varpi^{-1}e_1 + e_3)  \oplus \mc O_{k'} e_2 \oplus \mc O_{k'} \varpi e_3 \right ]$} node{$\bullet$};
        \draw (0,0)--(5,0);
        \draw (0,0)--(2.5,3);
        \draw (2.5, 3)--(5,0);
    \end{tikzpicture}
    \end{center}
	is fixed under the action of the Galois group, whereas only the bottom edge actually lies in $\mc B(G, k)$.
\end{ex}

\subsection{Fixed points in the compactification}

In this section, we prove that the situation for Galois-fixed points in the maximal Satake-Berkovich compactification of the building is no worse than that of the building, in the sense that Galois-fixed points in the compactification are limits of fixed points in the building. This is consistent with the observation that the boundary of the compactification is a union of Bruhat-Tits buildings of smaller rank.

\begin{thm}
If $k'/k$ is a finite Galois extension with group $\Gamma$, we have:
 $$(\ol{\mc B}_{\emptyset}(G, k'))^{\Gamma} = \ol{\mc B(G,k')^{\Gamma}},$$
 with the closure in the right-hand side taken in $\ol{\mc B}_{\emptyset}(G, k').$
\end{thm}

Note that, because $\Gamma$ acts continuously, we must have $\ol{\mc B(G, k')^{\Gamma}} \subset (\ol{\mc B}_{\emptyset}(G,k'))^{\Gamma}$.

The proof of the theorem therefore reduces to proving the reverse inclusion. That proof proceeds in three steps. A key observation is the fact that, if $P$ is a parabolic $k$-subgroup of $G$, $M$ a Levi factor of $P$ and $S$ a maximal split torus of $M_{k'}$, then we may map the apartment $A(S, k')$ to the building $\mc B(P_{\sss}, k')$ in two ways. One is given abstractly by Bruhat-Tits theory: There is a natural $M(k')\rtimes \Gamma$-equivariant identification $$\pi: \mc B^e(M, k') \simeq \mc B(P_{\sss}, k') \times \Hom(Z(M), \RR_{>0})$$ which restricts to a canonical projection $A(S, k') \epi A(\ol S, k')$, where $$\ol S = S/Z(M)\subset M/Z(M)\simeq P_{\sss}.$$ On the other hand, given a cocharacter $\lambda \in X_*(S)$ such that $P = P_G(\lambda)$, we prove below that, for each $x \in A(S, k')$, the limit $\lim\limits_{n \to \infty} \lambda^n \cdot x$ exists in $\ol{\mc B}_{\emptyset}(G, k')$ and lies in $A(\ol S, k')$, viewed as a subset of the boundary of $A(S, k')$, which defines a second map $A(S,k') \to A(\ol S, k')$. Moreover, under the above identifications, the two maps coincide.

The strategy is then the following: Given a Galois-fixed point $x \in \ol{\mc B}_{\emptyset}(G,k')^{\Gamma}$, lying in an apartment $A(\ol S, k')$ in the stratum $\mc B(P_{\sss}, k')$, we first establish that the parabolic subgroup $P$ is defined over $k$. Next, if $M$ is a Levi factor of $P$ that is defined over $k$ and $S$ the inverse image of $\ol S$ in $M$, we prove that there exists a Galois-fixed point $\tilde x_0 \in A(S, k')$ such that $\pi(\tilde x_0) = x$. Finally, we prove that if $\lambda \in X_*(S)$ is a $k$-rational cocharacter such that $P = P_G(\lambda)$, then $\lambda^n\cdot \tilde x_0$ is a sequence of Galois-fixed points that converges to $x$.

\begin{proof}

Let  $x \in (\ol{\mc B}_{\emptyset}(G,k'))^{\Gamma}$. Assume that $x \not \in \mc B(G, k')$. Then, by \cite[Lemma 4.12]{RTW1}, the Zariski closure of its stabiliser in $G(k')$ is a proper parabolic $k'$-subgroup $P$ of $G_{k'}$ and the boundary stratum of $x$ in $\ol{\mc B}_{\emptyset}(G, k')$ identifies with $\mc B(P_{\sss}, k')$. Because $x$ is $\Gamma$-invariant, so must be its stabiliser and the Zariski closure of its stabiliser. Hence, $P$ is defined over $k$.  Moreover, because the embedding $\mc B(P_{\sss}, k') \hookrightarrow \Bor(G_{k'})^{\an}$ is $\Gamma$-equivariant, as is clear on the definition \cite[4.1.1]{RTW1}, the induced action of $\Gamma$ on the boundary stratum of $x$ coincides with the action of $\Gamma$ on $\mc B(P_{\sss}, k')$ given by Bruhat-Tits theory applied to $P_{\sss}$. It then follows that $x \in \mc B(P_{\sss}, k')^{\Gamma}$.
Let $\ol S \subset P_{\sss}$ be a maximal $k'$-split $k'$-torus whose associated apartment $A(\ol S, k')$ contains $x$. 
Fix a Levi factor $M$ of $P$ that is defined over $k$ and let $S$ be the unique maximal $k'$-split $k'$-torus of $M$ that lifts $\ol S$.

Then, we have a $\Gamma$-equivariant and $M(k')$-equivariant identification $$\mc B^e(M, k') \simeq \mc B(P_{ss}, k') \times\Hom(Z(M), \RR_{>0}).$$
Furthermore, the projection $$\pi: \mc B^e(M, k') \epi \mc B(P_{ss}, k')$$ is $\Gamma$-equivariant and we have $\pi^{-1}(A(\ol S, k')) = A(S, k')$. Set $A = A(S, k')$.
In particular, the fibre $F = \pi^{-1}(x)$ is stable under the action of $\Gamma$ and contained in $A$ as an affine subspace.

Let $x_0 \in \pi^{-1}(x)$ and $\Gamma\cdot x_0$ be its orbit under $\Gamma$. We claim that the \textit{circumcenter} $\tilde x_0$ of $\Gamma \cdot x_0$, defined as the center of the unique ball in $\mc B^e(M, k')$ of minimal radius containing $\Gamma \cdot x_0$, is then fixed by $\Gamma$ and lies in the fibre $\pi^{-1}(x)$. The existence and uniqueness of the circumcenter of a bounded subset is a consequence of the CAT(0) property of the metric on $\mc B^e(M, k')$ -- see \cite[Chapter VI, Theorem 2]{Brown}. The uniqueness property and the fact that $\Gamma$ acts on $\mc B^e(M, k')$ by isometries immediately implies that $\tilde x_0$ is fixed under $\Gamma$. 

We prove that $\tilde x_0$ lies in the fibre $\pi^{-1}(x)$ in two steps: First, by proving that $\tilde x_0$ lies in the apartment $A$, then by proving that it lies in the fibre $F$. 

First, recall \cite[3A]{Brown} that there exists a $1$-Lipschitz retraction $$\rho: \mc B(G, k') \to A(S, k').$$
Let $r_{\min}$ be such that $\ol{B}(\tilde x_0, r_{\min})$ is the unique ball of minimal radius containing $\Gamma \cdot x_0$. Then, because $\rho$ leaves $\Gamma \cdot x_0$ invariant and is $1$-Lipschitz, we must have $\Gamma \cdot x_0 \subset \ol{B}(\rho(\tilde x_0), r_{\min})$. Because the ball of minimal radius containing the orbit is unique, we must have $\rho(\tilde x_0) = \tilde x_0$ or, in other words, $\tilde x_0 \in A(S, k')$. It follows that $\tilde x_0$ is the center of the minimal bounding ball for $\Gamma \cdot x_0$, this time viewed as a subset of the affine space $A$. Denote by $$p: A\to F$$ the orthogonal projection. Then, the image under $p$ of the minimal bounding ball $\ol B_{A}(\tilde x_0, r_{\min})$ is the ball $\ol B_{F}(p(\tilde x_0), r_{\min}) $. That ball still contains $\Gamma\cdot x_0$, hence the ball $\ol B_A(p(\tilde x_0), r_{\min})$ is a minimal bounding ball for $\Gamma \cdot x_0$. By uniqueness of the minimal bounding ball in Euclidean spaces, we now have $p(x_0) = x_0$, and therefore $x_0 \in F$, ie. $\pi(x_0) = x$.

Let $\lambda \in X_*(P)$ be a \underline{$k$-rational} cocharacter such that $P = P_G(\lambda)$. Then, for all $n \in \NN$, we have $$\lambda^n \cdot \tilde x_0 \in \mc B(G, k')^{\Gamma}$$ and we claim that $$\lambda^n \cdot \tilde x_0 \underset{n \to \infty}{\longrightarrow} x.$$

To prove this, we introduce some coordinates. The computation is ultimately simple but, unfortunately, requires a bit of work to set up.

The choice of a special point $o \in A$ determines a simplicial isomorphism $$A \simeq \Hom(X^*(S), \RR_{>0}).$$
Recall from \cite[Proposition 3.35]{RTW1} that the closure of the apartment $A$ in $\ol{\mc B}_{\emptyset}(G, k')$ is the compactification of $A$ with respect to the Weyl fan, that is, under our identification, the fan consisting of the cones $$\mf C(P_0) = \{\alpha \le 1, \alpha \in -\Phi(P_0, S)\},$$ with $P_0$ ranging among the minimal parabolic subgroups of $G_{k'}$ containing $S$.

In our case, letting $P_0$ be a minimal parabolic $k'$-subgroup of $P$ containing $S$, our initial point $x$ actually lies in the partial compactification $\ol{A}^{P_0}$ where only the cone $\mf C(P_0)$ is compactified. Denote by $_{k'}\Delta^-$ the basis of corresponding to $P_0^{\opp}$ and by $_{k'}\Phi^- = -\Phi(P_0, S)$ the system of negative roots associated to $P_0$. We have an identification $$\ol{A}^{P_0} \simeq \Hom_{\Mon}(\langle _{k'}\Phi^- \rangle, \RR_{\ge 0})$$ that is compatible with the previous description of $A$. 
Moreover, under this identification, the intersection $\mc B(P_{ss}, k') \cap \ol A^{P_0}$ corresponds to the locus of the points $z \in \ol A^{P_0}$ where $$\begin{array}{ccc} \langle \alpha, z\rangle = 0 &\text{if} & \alpha \in \Phi(R_u(P^{\opp}), S) \\ \langle \alpha, z\rangle > 0 & \text{if} & \alpha \in \Phi(M, S)\cap\, _{k'}\Phi^- \end{array},$$
that is, the apartment $A(\ol S, k') \subset \mc B(P_{ss}, k')$.
Finally, under this identification, the map $$\pi_{|A}: A(S, k') \epi A(\ol S, k')$$ given by Bruhat-Tits theory identifies with the projection $$\Hom(X^*(S), \RR_{>0}) \epi \Hom(X^*(\ol S), \RR_{>0})$$ coming from the projection $S \epi \ol S$ that identifies $X^*(\ol S)$ with the subset 
$$\{\lambda\}^{\perp} = \{\chi \in X^*(S), \langle \chi, \lambda \rangle = 1\}.$$

Because the point $\tilde x_0$ is a preimage of $x$ under $\pi$, we have $$\langle \alpha, \tilde x_0\rangle = \langle \alpha, x\rangle \text{ for all } \alpha \in \Phi(M, S) = \Phi(P_{ss}, \ol S).$$
Moreover, because $P = P_G(\lambda)$, we have for each $\alpha \in \Phi(G, S)$ (in multiplicative notation): $$\begin{array}{ccc} \langle \alpha, \lambda \rangle = 1 & \text{ if } & \alpha \in \Phi(M, S) \\ \langle \alpha, \lambda \rangle > 1 & \text{ if } & \alpha \in \Phi(R_u(P), S)\end{array}.$$
Consequently, we have for any $\alpha \in\, _{k'}\Phi^-$ and $n \in \NN$:
$$\begin{array}{ccccc} \langle \alpha, \lambda^n\cdot  \tilde x_0 \rangle & = & \langle \alpha, x\rangle & \text{if} & \alpha \in \Phi(M,S) \\ \langle \alpha, \lambda^n\cdot \tilde x_0 \rangle & = & \underbrace{\langle \alpha, \lambda\rangle^n}_{\underset{n \to \infty}{\longrightarrow} 0} \langle \alpha, \tilde x_0\rangle & \text{if} &  \alpha \in \Phi(R_u(P^{\opp}), S))\end{array}.$$

Whence, for each $\alpha \in\, _{k'}\Phi^-$ $$\langle\alpha, \lambda^n \cdot \tilde x_0\rangle \underset{n \to \infty}{\longrightarrow} \left\{\begin{array}{cc} \langle \alpha, x\rangle & \text{if } \alpha \in \Phi(M,S) \\ 0 = \langle \alpha, x\rangle & \text{if } \alpha \in \Phi(R_u(P^{\opp}, S)) \end{array}\right. .$$
Because $(_{k'}\Phi^- \cap\Phi(M, S)) \cup \Phi(R_u(P^{\opp}), S)$ spans $\langle _{k'}\Phi^-\rangle$ as a monoid, we deduce $$\langle\alpha, \lambda^n \cdot \tilde x_0\rangle \underset{n \to \infty}{\longrightarrow} \langle \alpha, x\rangle$$ for each $\alpha \in \langle_{k'} \Phi^-\rangle$ and therefore $$\lambda^n \cdot \tilde x_0 \underset{n \to \infty}{\longrightarrow} x.$$

\end{proof}

\begin{ex}
    We keep the notations from the previous example. We have observed above that, in the case when $k = \QQ_2, k' = \QQ_2(\sqrt 2), \varpi = \sqrt 2$, the vertex $$x = [\mc O_{k'}(\varpi^{-1}e_1 + e_3) \oplus \mc O_{k'}e_2 \oplus \mc O_{k'}\varpi e_3]$$ is a fixed point of $\mc B(G, k')$ under the Galois involution that does not lie in $\mc B(G, k)$. Following the template of the above proof, we use this fact to produce a fixed point in the boundary of $\ol{\mc B}(G, k')$ that does not lie in $\ol{\mc B}(G, k)$.
    Let $$\lambda: \begin{array}{ccc} \mbb G_{m,k'} & \to & \SL_{3,k'}\\ t & \mapsto & \begin{pmatrix}t^2 & 0 & 0 \\ 0 & t^{-4} & 0 \\ 0 & 0 & t^2\end{pmatrix}\end{array}.$$
    Then $\lambda$ is a $k$-rational cocharacter and we have, for each $n \in \NN$, $$\lambda^n \cdot x = [\mc O_{k'}(\varpi^{-1}e_1+e_3)\oplus \mc O_{k'}\varpi^{2n} e_2 \oplus \mc O_{k'}\varpi e_3].$$
    This sequence converges to a fixed point in the building $\mc B(P_{ss}, k') \subset \ol{\mc B}(G, k')$, where $$P = P_G(\lambda) = \begin{pmatrix} \ast & \ast & \ast \\ 0 & \ast & 0 \\ \ast & \ast & \ast \end{pmatrix}.$$
    The subgroup $P$ is the stabiliser of the plane $\langle e_1, e_3\rangle$ in $(k')^3$. The semisimple quotient $P_{ss}$ therefore naturally identifies with $\SL(k' e_1 \oplus k' e_3)$ and the building $\mc B(P_{ss}, k')$ can be interpreted as the set of homothety classes of lattices in $k'e_1 \oplus k' e_3$.
    
    Moreover, we have, under this identification $$\lim_{n \to \infty}\lambda^n \cdot x =  [\mc O_{k'}(\varpi^{-1}e_1 + e_3) \oplus \mc O_{k'}\varpi e_3] \in \mc B(P_{ss}, k')^{\Gamma} \setminus \mc B(P_{ss}, k).$$
\end{ex}